%% file: decomposition-jacobian.tex
\documentclass[reqno, 11pt, a4paper]{amsart}

\usepackage{fullpage}
\selectfont

\usepackage[normalem]{ulem}

\input{defs.tex}

\let\ENDPICTURE\endinput

\begin{document}

\title{Decomposing Jacobians via Galois covers}
\date{\today}

\begin{abstract}
  Let $\phi : X \to Y$ be a (possibly ramified) cover between two algebraic curves of positive genus. We develop tools that may identify the Prym variety of $\phi$, up to isogeny, as the Jacobian of a quotient curve $C$ in the Galois closure of the composition of $\phi$ with a well-chosen map $Y \to \P^1$. This method allows us to recover all previously obtained descriptions of a Prym variety in terms of a Jacobian that are known to us, besides yielding new applications. We also find algebraic equations for some of these new cases, including one where $X$ has genus $3$, $Y$ has genus $1$ and $\phi$ is a degree $3$ map totally ramified over $2$ points.
\end{abstract}

\thanks{We acknowledge support from the PICS JADERE (281036) from the French CNRS. The fourth author was moreover supported by a Juniorprofessurenprogramm of the Science Ministry of Baden-Württemberg. Part of the work for this paper was carried out at the workshop ``Arithmetic of Curves'', held in Baskerville Hall in August 2018. We would like to thank the organisers Alexander Betts, Tim Dokchitser, Vladimir Dokchitser, Céline Maistret and Beryl Stapleton, as well as the Baskerville Hall staff, for providing a great opportunity to concentrate on this project.}

\input{authorinfo.tex}

\maketitle

\input{introduction.tex}

\input{automorphisms.tex}
\input{c2c6.tex}

\input{algorithms.tex}
\input{implementation.tex}
\input{results.tex}

\input{families.tex}

\begin{appendix}
\input{tables.tex}

\end{appendix}

\bibliographystyle{alphaabbr}
\bibliography{synthbib}

\end{document}

%% file: defs.tex
\usepackage[utf8]{inputenc}
\usepackage[OT2,T1]{fontenc}
\usepackage[english]{babel}

\usepackage{rotating}

\usepackage{amsmath}
\usepackage{amsfonts}
\usepackage{amssymb}
\usepackage{amsthm}

\usepackage{mathtools}
\usepackage{leftidx}

\usepackage{mathrsfs}
\usepackage{listings}
\usepackage{relsize}

\usepackage[mathcal]{euscript}
\usepackage{bm}

\usepackage{enumitem}
\usepackage{multirow}

\usepackage[dvipsnames]{xcolor}
\usepackage{graphicx}
\usepackage{tikz-cd}
\usetikzlibrary{matrix}
\usepackage[all,cmtip]{xy}

\usepackage{mathrsfs}
\usetikzlibrary{arrows}

\usepackage{colortbl}
\definecolor{mygray}{gray}{0.92}

\usepackage[backref=page,pagebackref=true,hyperindex=true,bookmarks=true]{hyperref}

\theoremstyle{plain}
\newtheorem{theorem}{Theorem}[section]

\newtheorem{proposition}[theorem]{Proposition}
\newtheorem{lemma}[theorem]{Lemma}
\newtheorem{corollary}[theorem]{Corollary}

\newtheorem{algorithm}[theorem]{Algorithm}

\theoremstyle{definition}
\newtheorem{definition}[theorem]{Definition}
\newtheorem{problem}[theorem]{Problem}

\theoremstyle{remark}
\newtheorem{remark}[theorem]{Remark}
\newtheorem{example}[theorem]{Example}

\numberwithin{equation}{section}

\def\epsilon{\varepsilon}

\def\tilde{\widetilde}

\DeclareMathOperator{\Aut}{Aut}

\DeclareMathOperator{\im}{Im}

\DeclareMathOperator{\Jac}{Jac}

\def\C{\mathbb{C}}

\def\P{\mathbb{P}\,}

\def\xbar{\overline{x}}

\def\xbar{\overline{x}}

\def\QQbar{\overline{\mathbb{Q}}}

\usepackage{bm}

\iftrue
\newcommand{\Cr}[1]{\textcolor{red}{\sf CR:~[#1]}}
\newcommand{\Rl}[1]{\textcolor{blue}{\sf RL:~[#1]}}
\else
\newcommand{\Cr}[1]{}
\newcommand{\Rl}[1]{}
\fi

\newcommand{\Prym}{\operatorname{Prym}}

\hypersetup{
  pdfauthor   = {Lombardo, Lorenzo, Ritzenthaler, Sijsling},
  pdftitle    = {Decomposing Jacobians via Galois covers},
  pdfsubject  = {},
  pdfkeywords = {},
  colorlinks=true,
  breaklinks=true, urlcolor=blue, linkcolor=blue, citecolor=blue,
  bookmarksopen=true}


%% file: authorinfo.tex

\author[Lombardo]{Davide Lombardo}
\address{%
Davide Lombardo,
Dipartimento di Matematica,
Università di Pisa,
Largo Bruno Pontecorvo 5,
56127 Pisa,
Italia
}
\email{davide.lombardo@unipi.it}

\author[Lorenzo]{Elisa Lorenzo Garc\'ia}
\address{%
	Elisa Lorenzo Garc\'ia,
  Univ Rennes, CNRS, IRMAR - UMR 6625, F-35000
 Rennes, %
  France. %
}
\email{elisa.lorenzogarcia@univ-rennes1.fr}

\author[Ritzenthaler]{Christophe Ritzenthaler}
\address{%
	Christophe Ritzenthaler,
  Univ Rennes, CNRS, IRMAR - UMR 6625, F-35000
 Rennes, %
  France. %
}
\email{christophe.ritzenthaler@univ-rennes1.fr}

\author[Sijsling]{Jeroen Sijsling}
\address{
 Jeroen Sijsling,
 Institut für Reine Mathematik,
 Universität Ulm,
 Helmholtzstrasse 18,
 89081 Ulm,
 Germany
}
\email{jeroen.sijsling@uni-ulm.de}



%% file: introduction.tex

\section*{Introduction}

Let $k$ be an algebraically closed field of characteristic $0$, and let $X/k$ be a smooth projective irreducible curve of genus $g>0$. Already in the 19th century, complex geometers were interested in understanding the periods of $X$ in terms of periods of curves of smaller genera. In modern terms, one would like to decompose the Jacobian $\Jac (X)$ of $X$ as a product (up to isogeny) of powers of non-trivial simple abelian sub-varieties $A_i$, and then interpret the $A_i$ as Jacobians of suitable curves $C_i$. This is not possible for every curve $X$: Jacobians are generically simple, and even in those cases when $\Jac(X)$ does decompose there is no reason for the simple factors $A_i$ to be isogenous to Jacobians of curves. Indeed, while using a suitable isogeny allows us to assume that the $A_i$ are principally polarized, such abelian varieties are generically not Jacobians if $\dim (A_i) \ge 4$.

When the automorphism group of $X$ is non-trivial, one can often find some curves $C_i$ as above as quotients of $X$ by well-chosen subgroups of $G$, and in certain cases one even gets all of the $A_i$ in this way. This strategy has been employed many times, frequently in combination with the Kani-Rosen formula \cite{kanirosen}, to get (more or less explicit) examples of Jacobians whose isogeny factors are again Jacobians. For instance, when $g=2$ and $\# G>2$, this strategy always gives the full decomposition of $\Jac X$ \cite{gaudryschost}, and more examples have been worked out in \cite{paulhus1,paulhus2,jimenez}. In Section~\ref{sec:Galois} we give analogous results for $g=3$, both for hyperelliptic and non-hyperelliptic curves $X$. We show that the decomposition of the Jacobian of a generic curve with a given non-trivial automorphism group can indeed be obtained in terms of quotients of $X$ in all cases, except when
\begin{equation}\label{eq:excep}
  \text{$g (X) = 3$, $X$ non-hyperelliptic, $\Aut (X) \cong C_2$ or $\Aut (X) \cong C_6$.} \tag{$\star$}
\end{equation}
In this scenario, the quotient $X/G$ is an explicit genus $1$ curve $Y$, but (even up to isogeny) the complementary abelian surface $A$ cannot be interpreted as the Jacobian of a quotient of $X$. However, equations for a curve $C$ such that $A \sim \Jac(C)$ have been worked out also in this case, by a technique that we now briefly recall.

The aforementioned complementary abelian surface $A$ is an example of a \emph{Prym variety}. Recall that the (generalized) Prym variety of a cover $\pi : X \to Y$ is defined as the identity component $P:=\Prym(X/Y)$ of the kernel $\ker (\pi_{*})$ of the pushforward map $\pi_* : \Jac X \to \Jac Y$ induced by $\pi$. There is a significant body of literature dedicated to the description of $P$ in many special situations. The case when $\pi$ is an unramified cover of degree $2$ is especially beautiful and well-understood \cite{mumfordprym}. In this case $P$ inherits a principal polarization from $\Jac X$, and in some circumstances it has been described as the Jacobian of an explicit curve (see \cite[p.346]{mumfordprym}, \cite{levin} or \cite{bruin}).  In the situation \eqref{eq:excep}, the morphism $X \to X/C_2$ is \emph{not} \'etale, but it was shown in \cite{RiRo} that one could realize this map as the degeneration of a family of \'etale covers between curves of genera $5$ and $3$, and using the previously mentioned work \cite{bruin} it was possible to give an explicit equation for a curve $C$ such that $P\sim \Jac(C)$. This result is recalled in Section~\ref{sec:RiRoRedone}. It is however unclear how this idea may be extended to other types of covers.

The present paper was motivated by the desire to find an alternative proof of the previous result which could lead to generalizations. Our main sources of inspiration were Donagi's work on Prym varieties \cite{donagi}, based on Galois-theoretic considerations, and a specific construction by \cite{dalaljan} in the setting when $X \to Y$ is a cover of degree $2$ of a hyperelliptic curve branched at 2 points $Q_1$ and $Q_2$ (see also \cite[Th.4.1]{levin}). In this latter work, $P$ is realized as the Jacobian of a curve $C$ obtained as a quotient of the Galois closure $Z$ of $X\to Y \to \P^1$ by a well-chosen subgroup of the automorphism group of $Z$. Here, the cover from $Y$ to $\P^1$ is the hyperelliptic quotient, and generically $Q_1$ and $Q_2$ are not sent to the same point of $\P^1$. In the case \eqref{eq:excep} however, the cover $X \to Y$ has $4$ branched points, and one has to carefully choose a map from $Y \to \P^1$ which collapses $2$ of them. With this particular choice, we show in Section~\ref{sec:RiRoRedone} that $P$ is again realized up to isogeny as the Jacobian of an explicit quotient of the Galois closure $Z$ of the composed morphism $X \to Y \to \P^1$.

We then set out to see to what extent this somewhat miraculous construction could yield results for other classes of (not necessarily Galois) morphisms $X \to Y$. There are two main difficulties in carrying out this program. One is fundamental: given a cover $\pi_{X/Y} : X\to Y$, there is no a priori reason for the Prym variety to appear as the Jacobian of a quotient of a related Galois cover. In fact the general principally polarized abelian variety of dimension $4$ is known to be a Prym variety, and by \cite{tsimerman-no-jac} we know that there are $4$-dimensional abelian varieties over $\QQbar$ that are not $\QQbar$-isogenous to a Jacobian. The second difficulty lies in the choice of the morphism $Y \to \P^1$: the Galois closure $Z$ of the composition $X \to Y \to \P^1$ depends very strongly on this choice, and we did not find a general principle to guide us. In fact, as already mentioned, there is little previous work concerning Prym varieties when the degree of $\pi_{X/Y}$ is greater than $2$, and the present project is mainly exploratory.

The existing literature seems to have focused on two main cases: (1) the consideration of Galois or \'etale covers $X \to Y$ \cite{ortega,ortega2} and (2) a top-down approach, which starts with a curve $Z$ with large automorphism group $G$ and decomposes of $\Jac (Z)$ in terms of Prym varieties of subcovers $Z / H \to Z / G$ \cite{recillas-rodriguez} (see also the more complete  arXiv version \cite{recillas2003Prym}). Our approach combines aspects of these previous methods, in that it starts with a completely general map $X \to Y$ and finds candidates for the Prym variety in terms of the Galois closure of a suitable composition $X \to Y \to \P^1$. Finding examples for which this approach yields results is by no means straightforward: it is hard to get explicit equations for the curves, and the Galois groups of the composed maps rapidly attain prohibitive size.

In order to be able to analyze complicated situations, we use powerful tools from monodromy theory, which we implemented in the computer algebra system \textsc{Magma} \cite{Magma} (be aware that we use version 2.25-3 and that our programs do not always work with version 2.24-5). Specifying covers $X \to Y \to \P^1$ by their ramification structure, we can describe all possible monodromy types for the branched cover $X \to \P^1$, which in turn yields complete Galois-theoretic information on the Galois closure $Z \to \P^1$ of this map. The enumeration of possible monodromy types is a classical problem, often used in the setting of Galois covers \cite{breuer,paulhus3}. In Section~\ref{sec:AlgoTheory} we recall the relevant theory and show how to adapt it in our cases, when the covers $X \to Y$ and $X \to \P^1$ need not be Galois. A crucial tool in our applications is a beautiful result by Chevalley and Weil \cite{chevalleyweil, weil2}, by means of which one can identify, for a given Galois quotient $\pi_{Z/C} : Z \to C$, the image of $H^0(C,\Omega^1_C)$ by $\pi_{Z/C}^*$ inside $H^0(Z,\Omega^1_Z)$. Since we can similarly describe the images under pullback of $H^0(X,\Omega^1_X)$ and $H^0(Y,\Omega^1_Y)$, intersecting these subspaces allows us to decide if $\Jac C$ is isogenous to $\Prym(X/Y)$. Implementation details are given in Section~\ref{sec:implem} and the code can be found at \cite{code-decomp}.

By using this approach, we could (up to the limitations imposed by keeping the running time of our programs acceptable) recover all situations previously known in the literature, see Section~\ref{sec:results}. We also found some new cases. For example, consider a cover $X \to Y$ of degree $d$, where $Y$ is an elliptic curve, totally ramified over $2$ points $Q_1,Q_2$ of $Y$, and compose it with the map $Y \to \P^1$ which identifies $Q_1$ and $Q_2$. For $d = 3, 4$, and for some cases with $d = 5$ (see Table~\ref{tab:results-total}), we have been able to check within reasonable time that the abelian variety $\Prym(X/Y)$ is indeed isogenous to the Jacobian of a quotient $C$ of the Galois closure $Z$ of $X \to Y \to \P^1$. When $d=3$ and $X \to Y$ is a non-Galois cover, it turns out that $X$ and the corresponding Galois closure $Z$ of genus $5$ are hyperelliptic. Using this information, we were able to write down equations for $X$ and $Y$, and to find an explicit equation for a curve $C$ for which $\Jac (C) \sim \Prym(X/Y)$, see Section~\ref{sec:Families}. In another direction, we were able to generalize the example of \cite{bruin} (a genus $5$ \'etale cover of degree $2$ of a genus $3$ curve) to genus $g_X=2g+1$ \'etale covers of degree $2$ of genus $g_Y=g$ curves with $g=4,5$ or $6$, where $Y$ is a generic trigonal curve. In particular, for $g_Y=4$ we cover all generic non-hyperelliptic cases in this way.

Finally, we remark that our construction can also be seen through the lens of Belyi's theorem. Indeed, Belyi's result guarantees that all algebraic curves arise as covers of $\mathbb{P}^1$ ramified over just 3 points, and our approach usually consists in finding a morphism $Y \to \mathbb{P}^1$ such that the branch locus of $X \to Y \to \mathbb{P}^1$ is smaller than it would be for a generic choice of the map $Y \to \mathbb{P}^1$. In fact, this suggests that we are only scratching the surface, handling just the easiest of cases, and it might even be possible to \textit{always} recover the curves $C_i$, if they exist, by choosing a suitable morphism $Y \to \mathbb{P}^1$. The problem of decomposing the Prym variety of $X \to Y$ then becomes that of choosing an appropriate rational function on $Y$, a point of view which seems to be genuinely new. We think that these experiments and theoretical motivations are sufficiently intriguing for the relations between Prym varieties and Galois constructions to merit further study.

%% file: automorphisms.tex

\section{Decomposing Jacobians of curves of genus $3$ with non-trivial automorphism group}\label{sec:Galois}

Let $k$ be an algebraically closed field of characteristic $0$. We will implicitly assume that $k = \C$ throughout this article, especially in Section \ref{sec:AlgoTheory} when using the theory of covers. We do note that the results thus obtained will still be valid when the characteristic of $k$ does not divide the order of the intervening Galois groups --- for example, the results in the current section continue to apply when $k$ has finite characteristic strictly larger than $7$.

In this section we consider the decomposition of Jacobians of curves of genus $3$ induced by the action of their automorphism group. Most of these results are folklore. Note that this approach does not always yield the full decomposition of the Jacobian, nor can it guarantee that the higher-dimensional factors found in this way are irreducible.

\subsection{Hyperelliptic case}\label{sec:ExtraAutoHyp}
There is a stratification of the moduli space of hyperelliptic curves of genus 3 by their automorphism group (see \cite{bouw-thesis, LR11} and the references in the latter). The inclusions between the different strata are summarized in the following diagram, where we write $C_n$ for the cyclic group of order $n$, $D_n$ for the dihedral group of order $2n$, $S_n$ for the symmetric group over $n$ elements and $U_6$ and $U_8$ for certain groups with respectively $24$ and $32$ elements:

\begin{equation*}\label{eq:HyperellipticStrata}
\xymatrix{
  & & & C_2 \ar@/_3pc/@{-}[ddddlll] \ar@{-}[d] \ar@{-}[ddll] & & & & \dim = 5 \\
  & & & C_2^2 \ar@{-}[dd] \ar@{-}[lldd] \ar@{-}[dr] & & & & \dim = 3\\
  & C_4  \ar@{-}[d] & & & C_2^3 \ar@{-}[dr] &  &  & \dim = 2\\
  & C_2 \times C_4 \ar@{-}[drr] \ar@{-}[d ]& & D_6 \ar@{-}[dll] \ar@{-}[drr] &  & C_2 \times D_4 \ar@{-}[dll] \ar@{-}[d]   & & \dim = 1 \\
  C_{14} & U_6 &  & U_8 &  & C_2 \times S_4 & & \dim = 0\\
}
\end{equation*}

\begin{proposition}\label{GaloisHyper}
  Suppose $X$ is a hyperelliptic curve of genus 3 whose automorphism group contains a group $G$ appearing in the previous diagram. Then the Jacobian of $X$ decomposes up to isogeny as in Table \ref{table:Hyperelliptic}.
\end{proposition}

{\tiny
\begin{center}
\begin{table}[ht]
\begin{tabular}{|c|c|c|c|}
\hline
$G$  & $X : y^2=f(x)$ & $\Jac(X) \sim \prod_{i \in I}  \Jac(C_i)$ & Curves $C_i$  \\ \hline
$C_2$ & & & $\Jac(X)$ generically simple \\
$C_2^2$ &  $x^8+a x^6+b x^4+ c x^2+1$ & $I=[1,2]$ & $\begin{cases}  C_1: & y^2=x^4+a x^3+b x^2+cx+1, \\  C_2: & y^2=x (x^4+a x^3+b x^2+cx+1) \end{cases}$ \\
$C_2^3$ & $x^8+a x^6+ b x^4+ a x^2+1$ & $I=[1,2,3]$ &  $\begin{cases}  C_1 :& y^2=x^4+ a x^3+bx^2+ax +1, \\ C_2:&  y^2=x^4+(a-4) x^2-2a+b+2 \\  C_3: & y^2=x^4+(a+4)x^2+2a+b+2 \end{cases}$ \\
$C_4$ & $x(x^2-1)(x^4+ax^2+b)$ & & $\begin{array}{c} \Jac(X) \text{ generically simple} \\ \text{with endomorphism algebra } \mathbb{Q}(i) \end{array}$ \\
$C_2 \times C_4$ &  $x^8+ax^6-ax^2-1=(x^4-1)(x^4+ax^2+1)$ & $I=[1,2]$ &   $\begin{cases}  C_1 :& y^2 = (x^2-1)(x^2+ax+1) \\ C_2 : & y^2 = x(x^2-1)(x^2+ax+1) \end{cases}$  \\
$D_6$ & $x(x^6+ax^3+1)$ &  $I=[1,2,2]$ &   $\begin{cases}  C_1 :& y^2=x(x^2+ax+1) \\ C_2 :& y^2=x^3-3x+a  \end{cases}$  \\
$C_2 \times D_4$ & $x^8+ax^4+1$ & $I=[1,2,2]$ & $\begin{cases}  C_1 :& y^2=x^4+a x^2+1 \\ C_2 :& y^2=x^4-4x^2+(a+2)  \end{cases}$  \\
$C_{14}$ & $x^7-1$ & &  $\Jac(X)$ simple with endomorphism algebra $\mathbb{Q}(\zeta_7)$ \\
$U_6$ & $x(x^6+1)$ &  $I=[1,1,1]$ & $C_1 : y^2=x^3+x$, i.e., $j=1728$ \\
$C_2 \times S_4$ & $x^8+14x^4+1$ & $I=[1,1,1]$ & $C_1 : y^2 = x^3 + x^2 - 4x - 4$, i.e., $j=2^4\cdot 3^{-2}\cdot 13^3$ \\
$U_8$ & $x^8+1$ &  $I=[1,2,2]$ &  $\begin{cases}  C_1 :&  y^2=x^4+1,\text{ i.e., }j=1728 \\ C_2 : & y^2=x^4-4x^2+2,\text{ i.e., }j=2^6\cdot5^3 \end{cases}$ \\
\hline
\end{tabular}
\caption{Decomposition of Jacobian: hyperelliptic case}
\label{table:Hyperelliptic}
\end{table}
\end{center}
}

\subsection{Non-hyperelliptic case}

A similar analysis can be carried out in the non-hyperelliptic case, with the notable exception of the group $C_2$ and its specialization $C_6$ which shall be reviewed in Section~\ref{sec:RiRoRedone}. There is a stratification of the moduli space of non-hyperelliptic genus 3 curves according to their automorphism group (see \cite[2.88]{henn},~\cite[p.62]{vermeulen},~\cite{MSSV02},~\cite{bars}
and~\cite{dolgacag}; the groups $G_i$ are certain groups of order $i$). The inclusions between the different strata are summarized in the following diagram:
\begin{equation*}\label{eq:Non-HyperellipticStrata}
\xymatrix{
  & & & \{\operatorname{id}\} \ar@{-}[d] \ar@{-}[dddll] & & &  \dim = 6\\
  & & & C_2 \ar@{-}[dddll] \ar@{-}[d] \ar@{-}[ddrr] & & &  \dim = 4\\
  & &  & C_2^2 \ar@{-}[d] & & &  \dim = 3\\
  & C_3 \ar@{-}[d] \ar@{-}[ddl] & & D_4 \ar@{-}[ld] \ar@{-}[rd] & & S_3 \ar@{-}[ld] &  \dim = 2 \\
  & C_6\ar@{-}[d] &    G_{16} \ar@{-}[ld] \ar@{-}[rd] & & S_4 \ar@{-}[ld] \ar@{-}[rd]  & & \dim = 1\\
  C_9 & G_{48} & & G_{96} & & G_{168} &  \dim = 0
}
\end{equation*}

\begin{proposition}\label{GaloisNon-Hyper}
  Suppose $X$ is a non-hyperelliptic curve of genus 3 whose automorphism group contains a group $G$ appearing in the previous diagram. Then the Jacobian of $X$ decomposes up to isogeny as in Table \ref{table:NonHyper}.
\end{proposition}

{\tiny
\begin{center}
\begin{table}[h]
\begin{tabular}{|c|c|c|c|}
\hline
$G$  & $X : F(x,y,z)=0$ & $\Jac(X) \sim \prod_{i \in I}  \Jac(C_i)$ & \text{Curves} $C_i$  \\ \hline
$C_2$ & & & see Section~\ref{sec:RiRoRedone}  \\
$C_2^2$ &  $x^4+y^4+z^4+r x^2 y^2+s y^2 z^2+t z^2 x^2$ & $I=[1,2,3]$ & $\begin{cases}  C_1: & y^2=(1/4r^2-1) x^4 + (1/2rs-t) x^2 + (1/4s^2-1) , \\  C_2: &  y^2=(1/4s^2-1) x^4 + (1/2st-r) x^2 + (1/4t^2-1), \\ C_3 &  y^2=(1/4t^2-1) x^4 + (1/2tr-s) x^2 + (1/4r^2-1), \end{cases}$ \\
$C_3$ & $x^3 z +y(y-z)(y-r z)(y-sz)$  &  &   $\begin{array}{c} \Jac(X) \text{ generically simple} \\ \text{with endomorphism algebra } \mathbb{Q}(\sqrt{-3}) \end{array}$  \\
$D_4$ & $x^4+y^4+z^4+r x^2yz+s y^2 z^2$ & $I=[1,2,2]$ & $\begin{cases}  C_1: & y^2= x^4 +(r^2/4-s) x^2 +1, \\  C_2 : & y^2 = (-s-2+r^2/4) x^4-2r x^2-s+2 \end{cases}$ \\
$S_3$ & $x(y^3+z^3)+y^2z^2+r x^2 yz+s x^4$ & $I=[1,1,2]$ &   $\begin{cases}  C_1 :& y^2 = -x^3+9/4x^2-3/2rx+r^2/4-s \\ C_2 :& y^2=x^4+2 r x^3+(r^2-4s)x^2-s x \end{cases}$  \\
$C_6$ & $x^3 z+y^4+r y^2 z^2+z^4$ &  $I=[1,2]$ & See Section \ref{sec:RiRoRedone}
\\
$G_{16}$ & $x^4+y^4+z^4+r y^2 z^2$ & $I=[1,2,2]$ & $\begin{cases}  C_1 :& y^2=x^4 -r x^2+1 \\ C_2 :&  y^2=(-r-2) x^4-r+2  \end{cases}$  \\
$S_4$ & $x^4+y^4+z^4+r(x^2y^2+y^2z^2+z^2x^2)$ & $I=[1,1,1]$ & $C_1:  y^2=(1/4r^2-1) x^4 + (1/2r^2-r) x^2 + (1/4r^2-1) $ \\
$C_9$&  $x^3y+y^3z+z^4$ &  & $\Jac(X)$ simple with endomorphism algebra $\mathbb{Q}(\zeta_9)$ \\
$G_{48}$ & $x^4+(y^3-z^3)z$ & $I=[1,2,2]$ & $\begin{cases}  C_1 :& y^2=x^3+1 \\ C_2 :& y^2=x^3+x   \end{cases}$  \\
$G_{96}$ & $x^4+y^4+z^4$ &  $I=[1,1,1]$ &  $C_1 : y^2=x^3+x$ \\
$G_{168}$ & $x^3 y+ y^3 z +z^3 x$ & $I=[1,1,1]$ & $\begin{array}{c} C_1 : y^2 + xy + y = x^3 - x^2 - 2680x + 66322, \\
\text{i.e.,} j=-3375 \end{array}$ \\
\hline
\end{tabular}
\caption{Decomposition of the Jacobian: non-hyperelliptic case}
\label{table:NonHyper}
\end{table}
\end{center}
}

%% file: c2c6.tex

\section{Plane quartics with automorphism group $C_2$ or $C_6$}\label{sec:RiRoRedone}

Let $X/k$ be a non-hyperelliptic curve of genus 3 with automorphism group $C_2$. The action of the automorphism induces a map $\pi : X \to Y$ of degree $2$, where $Y$ is an elliptic curve. Hence we know that $\Jac X \sim Y \times A$, but $A$ is not the Jacobian of a subcover of $X$. Indeed, the Riemann-Hurwitz formula shows that any morphism $X \to C$ with $g(C)=2$ must be of degree 2, hence should come from another involution of $X$. The problem of describing $A$ up to isogeny as the Jacobian of an explicit curve $C$ of genus $2$ was solved in \cite{RiRo} by relying on a suitable deformation of $\pi$ to an \'etale cover between curves of genera 5 and 3. This result is recalled below. Since every non-hyperelliptic genus 3 curve with an involution can be written in the form \eqref{eq:PlaneQuarticWithInvolution}, this completes the tables of Section~\ref{sec:Galois}.

\begin{proposition}[Ritzenthaler-Romagny \cite{RiRo}]\label{prop:RiRo}
Let $X$ be a smooth, non-hyperelliptic genus $3$ curve defined by
\begin{equation}\label{eq:PlaneQuarticWithInvolution}
  X : y^4 - h(x,z) \, y^2 + f(x,z) \,  g(x,z)=0
\end{equation}
in $\P^2_k$, where $$f=f_2 x^2+ f_1 x z + f_0 z^2, \quad g=p_2 x^2+g_1 x z + g_0 z^2, \quad h = h_2 x^2 + h_1 xz +h_0 z^2$$ are homogeneous polynomials of degree $2$ over a field $k$ of characteristic different from $2$.  The involution $(x:y:z) \mapsto (x:-y:z)$ induces a cover $\pi$ of degree $2$ of the genus $1$ curve  $$Y : y^2 - h(x,z) \, y + f(x,z) \, g(x,z)=0$$ in the weighted projective space $\P{(1,2,1)}$.
Let $$M = \begin{bmatrix} f_2 & f_1 & f_0 \\ h_2 & h_1 & h_0 \\ p_2 & g_1 & g_0 \end{bmatrix}$$
and assume that $M$ is invertible. Let $$M^{-1} = \begin{bmatrix} a_1 & b_1 & c_1 \\ a_2 & b_2 & c_2 \\ a_3 & b_3 & c_3 \end{bmatrix}.$$ Then  $\Jac(X) \sim \Jac(Y) \times \Jac(C)$ with $C : y^2 = b \cdot (b^2-ac)$ in $\P{(1,3,1)}$ where
$$a=a_1+2 a_2 x+ a_3 x^2, \quad b=b_1 + 2 b_2 x + b_3 x^2, \quad c=c_1+2 c_2 x+ c_3 x^2.$$
\end{proposition}

In the special case when the automorphism group of $X$ is $C_6$, it can be realized as a plane quartic
\begin{equation*}
  X : x^3z+y^4+ry^2z^2+z^4 = 0
\end{equation*}
for some $r \in k$, and we find $\Jac(X) \sim C_1 \times \Jac(C_2)$ with
\begin{equation*}
  \begin{cases}  C_1 :& y^2 = -x^3 +r^2/4-1 \\ C_2 : & y^2=(x^2 - 2x - 2) (x^4 - 4x^3 + (-2 r^2 + 8)x  - r^2 + 4) \end{cases}
\end{equation*}

In the next subsection we explain a different approach to handle the case of non-hyperelliptic curves with automorphism group $C_2$. This will serve as motivation for the generalization discussed in Section~\ref{sec:AlgoTheory}.

\subsection{A Galois approach}
Let $X$ be as in the previous section, that is, a non-hyperelliptic genus $3$ curve with an involution. The corresponding quotient is a curve $Y$ of genus $1$, and the morphism $\pi_{X/Y} = \pi : X \to Y$ of degree $2$ is branched over $4$ distinct points $Q_1,Q_2, Q_3$ and $Q_4$. Let us consider a morphism $\pi_{Y/\mathbb{P}^1} : Y \to \P^1$ which maps $Q_1$ and $Q_2$ to the same point $[\beta,1] \in \P^1$ with $\beta \ne 0$. Choosing an origin on $Y$, and thereby giving it the structure of an elliptic curve, this morphism can be constructed by taking the quotient of the elliptic curve $Y$ by the involution $P \mapsto Q_1+Q_2-P$. Composing with an automorphism of $\mathbb{P}^1$ if necessary, we can and will assume that $\pi_{Y/\mathbb{P}^1}(Q_3)=[0,1]$. Additionally, we write $\pi_{Y/\mathbb{P}^1}(Q_4)=[\gamma,1]$.

\begin{remark}
  Consider the special case $\gamma=0$, that is, the morphism $Y \to \mathbb{P}^1$ identifies $Q_1$ with $Q_2$, as well as $Q_3$ with $Q_4$. The methods developed later on in Section \ref{sec:AlgoTheory} will enable us to show that this happens if and only if the composite map $X \to \mathbb{P}^1$ is Galois, with Galois group $C_2^2$ (see Table \ref{tab:results-g3}). In this case $\Aut(X)$ contains a copy of the Klein $4$-group $C_2^2$, and we see from Table \ref{table:NonHyper} that $\Jac(X)$ decomposes as the product of three elliptic curves, each of which is a quotient of $X$.
\end{remark}

From now on we  restrict to the case $\gamma \neq 0$. By \cite{enolskii,levin} we have the following equations: we may write $Y : y^2=f(t) =(t-\alpha_1)(t-\alpha_2)(t-\alpha_3)$ and
\begin{equation*}
  X = \begin{cases} y^2=& f(t) \\ x^2=& (t-\beta) (p_2(t) + y)\end{cases}
\end{equation*}
where $p_2$ is a polynomial of degree $2$ such that $p_2(t)^2- f(t)= t(t-\gamma) p_1(t)^2$ with $p_1(t)$ a polynomial of degree $1$.

Let $Z \to \mathbb{P}^1$ be the Galois closure of $X \to \mathbb{P}^1$. The Galois group of $Z/\mathbb{P}^1$ is isomorphic to $D_4$, the dihedral group on $4$ elements. We write $D_4=\langle r,s \bigm\vert r^4=s^2=1, sr=r^3s\rangle$, assuming (as we may) that $s$ to is the non-central element of order 2 such that $Z/\langle s \rangle \cong X$. Let $\xbar$ be any root of  $\xbar^2= (t-\beta)(p_2(t)-y)$. We then see that $Z$ is the smooth projective curve an affine part of which is given by
\begin{equation*}
    \begin{cases}
    y^2 =& f(t), \\
    x^2=& (t-\beta) (p_2(t)+y), \\
    \xbar^2= & (t-\beta) (p_2(t)-y).
  \end{cases}
\end{equation*}
Since $X$ corresponds to the quotient of $Z$ by $s$, we know that $s$ sends $(x, \xbar, y)$ to $(x, -\xbar, y)$. We can choose $r$ to be $(x,\xbar,y) \mapsto (\xbar, -x, -y)$. Direct inspection of the subgroup lattice of $D_4$ implies that the maps $Z \to X \to Y \to \mathbb{P}^1$ fit into a larger diagram of maps of degree $2$:

\begin{equation}\label{eq:Degree2GaloisDiagram}
  \xymatrix{
    & Z \ar[ld]_{\pi_1} \ar[d] \ar[rd]^{\pi_2}
    & \\
    X=Z/\langle s \rangle \ar[d]
    & Z/\langle r^2 \ar[ld] \ar[d] \ar[rd] \rangle & C= Z/\langle sr \rangle \ar[d]
    \\
    Y = Z / \langle r^2, s \rangle \ar[rd] & Z/\langle r \rangle \ar[d] & Z/\langle r^2, sr \rangle \ar[ld] \\
    & \mathbb{P}^1 = Z/D_4
  }
  \quad
  \quad
  \quad
  \xymatrix{
    & g=7 \ar[ld] \ar[d] \ar[rd]
    & \\
    g=3 \ar[d]
    & g=3 \ar[ld] \ar[d] \ar[rd] & g=2 \ar[d]
    \\
    g=1 \ar[rd] & g=2 \ar[d] & g=0 \ar[ld] \\
    & g=0
  }
\end{equation}

Knowing the action of $s$ and $r$ explicitly allows us to work out equations for the various quotients in the previous diagram and to compute their genera using the Riemann-Hurwitz formula. We will be mainly interested in $C=Z/\langle sr \rangle$. Consider the $sr$-invariant functions $v=x+ \xbar$ and $\displaystyle w= \frac{x \xbar}{(t-\beta) p_1(t)}$. Note that the invariant function $z:=y(x-\xbar)$ also lies in the function field $k(v,w,t)$, since $vz=y(x^2-\xbar^2)=2f(t)(t-\beta)$. The relations between $v$, $w$ and $t$ describe the quotient curve
\begin{equation}\label{eq:Y}
  C = \begin{cases} v^2 =&  2 (t-\beta) (p_2(t) +w p_1(t)),\\ w^2 =& t (t- \gamma) \end{cases}
\end{equation}
which is indeed a cover of $\P^1$ of degree $4$, as it is a cover of degree $2$ of a conic that is in turn a cover of degree $2$ of $\P^1$.

The second equation in \eqref{eq:Y} describes a conic with a rational point, which may be parametrized as $(t,w)=\left( \frac{\gamma}{1-u^2}, \frac{\gamma u}{1-u^2} \right)$. Replacing this parametrization in the first equation and setting $s := (1-u^2)^3v$ we then get the hyperelliptic model
\begin{equation}\label{eq:ModelGenus2Curve}
  s^2 = 2 \left( \gamma - \beta(1-u^2) \right) \left( (1-u^2)^2 p_2 (\frac{\gamma}{1-u^2}) + \gamma u (1-u^2) p_1 (\frac{\gamma}{1-u^2}) \right).
\end{equation}

\begin{remark}
  The model \eqref{eq:ModelGenus2Curve} is smooth and defines a curve of genus 2. Indeed, one may check that (under our assumptions $\beta \neq 0, \gamma \neq 0, \beta \neq \gamma$) the irreducible factors of the discriminant of the polynomial on the right hand side are also factors of either $\operatorname{disc}(f)$ or $\operatorname{Res}_t (f(t), (t-\beta)p_2(t))$. This shows that \eqref{eq:ModelGenus2Curve} is smooth, because $\operatorname{disc}(f)=0$ (resp. $\operatorname{Res}_t (f(t), (t-\beta)p_2(t))=0$) would imply that $Y$ (resp.\ $X$) is not smooth.
\end{remark}

We now aim to show that the Prym variety of the cover $X \to Y$ is isogenous to $\Jac(C)$ (Theorem \ref{thm:RiRoRe}). In order to do so, we begin by investigating the action of $D_4 \subset \Aut(Z)$ on the space of regular differentials $H^0(Z,\Omega^1_Z)$. We will freely use some results that will be discussed in general in Section \ref{sec:AlgoTheory}, see in particular Theorem \ref{thm:PusforwardPullbackDifferentials}. Recall that the character table of $D_4$ is as follows:

\begin{center}
  \begin{tabular}{c|c|c|c|c|c}
    & $\{\operatorname{id}\}$ & $\{r^2\}$ & $\{s,sr^2\}$ & $\{sr, sr^3\}$ & $\{r, r^3\}$ \\
    \hline

    $(1)$   & 1 & 1 & 1 & 1 & 1 \\
    $V_1$   & 1 & 1 &$-1$ &$-1$ & 1 \\
    $V_2$   & 1 & 1 & 1 &$-1$ &$-1$ \\
    $V_3$   & 1 & 1 &$-1$ & 1 &$-1$\\
    $(2)$ & 2 & $-2$ & 0 & 0 & 0
  \end{tabular}
\end{center}

Note in particular that $r^2$ acts trivially on the 1-dimensional representations $V_1$, $V_2$, $V_3$ and as $-1$ on $(2)$, while the fixed subspace in $(2)$ of each of the symmetries $s,sr,sr^2,sr^3$ is 1-dimensional.

\begin{lemma}\label{lemma:ActionD4RegularDifferentials}
  We have $H^0(Z,\Omega^1_Z) \cong V_1^{\oplus 2} \oplus V_2 \oplus (2)^{\oplus 2}$ as representations of $D_4$.
\end{lemma}

\begin{proof}
  Write $H^0(Z,\Omega^1_Z) \cong (1)^{\oplus e_0} \oplus V_1^{\oplus e_1} \oplus V_2^{\oplus e_2} \oplus V_3^{\oplus e_3} \oplus (2)^{\oplus e_4}$ as representations of $D_4$. Let $H$ be any subgroup of $G$. One has
  \begin{equation}
    H^0(Z,\Omega_Z)^H \cong H^0(Z/H,\Omega_{Z/H}),
  \end{equation}
  which implies that the dimension of the subspace of $H^0(Z,\Omega^1_Z)$ fixed by $H$ is the genus of $Z/H$. Applying this to $H=G$, and observing that $Z/G \cong \mathbb{P}^1$ has genus $0$, we obtain that $H^0(Z,\Omega^1_Z)$ does not contain any copy of the trivial representation, i.e., $e_0=0$. Applying the same argument with $H=\langle r^2\rangle$ one obtains $g(Z/H)=3 = \dim H^0(Z,\Omega^1_Z)^H$, and since $r^2$ acts trivially on $V_1, V_2, V_3$ and without fixed points on $(2)$ this implies $3=e_1+e_2+e_3$. We also have the condition $e_1+e_2+e_3+2e_4=\dim H^0(Z,\Omega^1_Z)=7$, so -- combining the last two equations -- we obtain $e_4=2$. Finally, the conditions
  \begin{equation}
    3 = g(Z/\langle s \rangle) = \dim H^0(Z,\Omega^1_Z)^{\langle s \rangle} = e_2+e_4
  \end{equation}
  and
  \begin{equation}
    2 = g(Z/\langle sr \rangle) = \dim H^0(Z,\Omega^1_Z)^{\langle sr \rangle} = e_3 + e_4
  \end{equation}
  imply $e_2=1, e_3=0$ and therefore $e_1=2$.
\end{proof}

\begin{lemma}\label{lem:genus2Correspondence}
  The correspondence
  \begin{equation}
    \xymatrix{
      & Z \ar[ld]_{\pi_1} \ar[rd]^{\pi_2} & \\
      X=Z/\langle s \rangle &  & C = Z/\langle sr \rangle
      }
      \quad\quad
    \xymatrix{
      & g=7 \ar[ld] \ar[rd] & \\
      g=3 & & g=2 \\
    }
  \end{equation}
  induces a homomorphism of abelian varieties $\Jac(Z/\langle sr \rangle) \to \Jac(X)$ with finite kernel. In particular, $\Jac(Z/\langle sr \rangle)$ is a factor of $\Jac(X)$ in the category of abelian varieties up to isogeny.
\end{lemma}

\begin{proof}
  We consider the action of this correspondence on regular differentials and determine the image of
  \begin{equation}
    \pi_{1*} \pi_2^* : H^0(C, \Omega^1_C) \to H^0(X, \Omega^1_X).
  \end{equation}
  The image of $\pi_2^*$ is the $sr$-invariant subspace of $H^0(Z,\Omega^1_Z)$; given our description of $H^0(Z,\Omega^1_Z)$ as a $D_4$-representation, we see that this is precisely the $sr$-invariant subspace in $(2)^{\oplus 2}$. Identifying $H^0(X,\Omega^1_X)$ with $H^0(Z,\Omega_Z^1)^{\langle s \rangle}$, the map
  \begin{equation}
    \pi_{1*} : H^0(Z,\Omega_Z^1) \to H^0(X,\Omega_X^1) \cong H^0(Z,\Omega_Z^1)^{\langle s \rangle}
  \end{equation}
  is given by $\omega \mapsto \omega + s^*\omega$ . Since the structure of the $2$-dimensional representation $(2)$ shows that the map $(1+s)$ is injective on its $sr$-invariant subspace, we obtain that $\pi_{1*}$ is injective on the image of $\pi_2^*$. This implies that the image of $\pi_{1*} \pi_2^*$ is 2-dimensional, which in turn means that the image of $\Jac(Z/\langle sr \rangle) \to \Jac(X)$ is 2-dimensional as claimed.
\end{proof}

\begin{theorem}\label{thm:RiRoRe}
  The Jacobian of $X$ decomposes up to isogeny as
  \begin{equation}
    \Jac(X) \sim Y \times \Jac(Z/\langle sr \rangle).
  \end{equation}
  As a consequence, $\Jac(Z/\langle sr \rangle)$ is isogenous to the Prym variety of $\pi:X \to Y$, and a nontrivial map $\Jac(Z/\langle sr \rangle) \to \Jac(X)$ is induced by the correspondence $Z$ in \eqref{eq:Correspondence}.
\end{theorem}

\begin{proof}
  In the light of Lemma \ref{lem:genus2Correspondence} it suffices to prove that the subspaces $\pi_{X/Y}^* H^0(Y,\Omega^1_Y)$ and $\pi_{1*} \pi_2^* H^0(C, \Omega^1_{C})$ of $H^0(X, \Omega^1_X)$ generate this vector space, or equivalently (by dimension considerations) that they intersect trivially. Since $\pi_1^* : H^0(X,\Omega^1_X) \to H^0(Z,\Omega^1_Z)$ is injective, it suffices to prove that they intersect trivially after pullback to $H^0(Z,\Omega^1_Z)$. One can describe the subspaces $\pi_1^* \pi_{X/Y}^* H^0(Y,\Omega^1_Y)$ and $\pi_1^*\pi_{1*}\pi_2^*H^0(C, \Omega^1_C)$ in terms of the action of $D_4$: according to Diagram \eqref{eq:Degree2GaloisDiagram} and Lemma \ref{lemma:ActionD4RegularDifferentials}, $\pi_1^* \pi_{X/Y}^* H^0(Y,\Omega^1_Y) = H^0(Z,\Omega^1_Z)^{\langle r^2,s \rangle} = V_2$, while
  \begin{equation*}
    \pi_1^*\pi_{1*}\pi_2^*H^0(C, \Omega^1_C) = (1+s) H^0(Z,\Omega^1_Z)^{\langle sr \rangle}.
  \end{equation*}
  It now suffices to note that $sr$ has no nonzero fixed points in $V_1^{\oplus 2} \oplus V_2$, so $H^0(Z,\Omega^1_Z)^{\langle sr \rangle}$ is contained in $(2)^{\oplus 2}$. Since $(2)^{\oplus 2}$ is a subrepresentation of $H^0(Z,\Omega^1_Z)$ it follows that also $(1+s)H^0(Z,\Omega^1_Z)^{\langle sr \rangle}$ is contained in $(2)^{\oplus 2}$, hence it does not intersect $H^0(Z,\Omega^1_Z)^{\langle r^2,s \rangle}$ as claimed.
  We conclude as desired that $\pi_1^* \pi_{X/Y}^* H^0(Y,\Omega^1_Y)$ and $\pi_2^*(C, \Omega^1_C)$ together generate $H^0(Z,\Omega^1_Z)^{\langle s \rangle}=H^0(X,\Omega^1_X)$.
\end{proof}

Theorem \ref{thm:RiRoRe} recovers Proposition \ref{prop:RiRo} and also clarifies the nature of a correspondence between $C$ and $X$. In addition, notice that the curve $C$ described in Proposition \ref{prop:RiRo} depends on the choice of a factorization $f(x,z)g(x,z)$ of a certain polynomial of degree $4$ as the product of two quadratics. Note that the zero locus of $f(x,z)g(x,z)$ on $Y$ describes precisely the branch locus of $X \to Y$. In our new approach, the choice of factorization can be reinterpreted as the choice of the two points $Q_1, Q_2$ that are contracted by the morphism $\pi_{Y/\mathbb{P}^1}$.

\begin{remark}
  In \cite{RiRo}, the aforementioned choice of a partition of $4$ points into $2$ pairs is clearly symmetric in the pairs. By contrast, in this new approach the choice is highly asymmetric since $2$ points are contracted and the other $2$ are not.
\end{remark}

%% file: algorithms.tex

\newcommand{\minus}{-}

\section{An algorithmic approach via group theory}\label{sec:AlgoTheory}

Our purpose in this section is to generalize the previous discussion to more complicated cases, for which explicit equations are not available. The proof of Lemma \ref{lemma:ActionD4RegularDifferentials} relied strongly on the fact that we could compute the genus of any quotient of $Z$ by direct inspection of the equations of the curves and of the action of automorphisms. In general it is more difficult to get such information explicitly, so in this section we explain how we may reverse the process: we first describe the action of $\Aut(Z \to \mathbb{P}^1)$ on $H^0(Z,\Omega^1_Z)$ (Paragraph \ref{subsec:GModuleStructure}), and subsequently rely on this information to completely describe the morphisms between Jacobians of curves obtained as quotients of $Z$ (Paragraph \ref{sec:MapsJacobians}). The method has its roots in the theory of monodromy actions for branched covers of curves. While developing the main notions of this theory below, we show how it can be combined with the description of the aforementioned action, and also give some explicit references for useful statements in this context, in particular Theorem \ref{thm:GaloisClosure}.

\subsection{Preliminaries on ramification and monodromy}

In this section we fix our notation and conventions for describing the ramification of a morphism between smooth projective curves over $\mathbb{C}$. We will freely use without further mention the fact that the category of such curves is equivalent to the category of Riemann surfaces, and assume that all our curves are connected.
We will find it useful to introduce the following definition:

\begin{definition}
  Let $\varphi : X \to Y$ be a morphism of smooth projective curves over $\mathbb{C}$ and let $B=(b_1,\ldots,b_n)$ be a fixed ordered subset of $Y$ which contains the branch locus of $\varphi$. For $b \in Y$, let $\varphi^{-1}(b)=\{a_1,\ldots,a_k\}$ be the fiber of $\varphi$ above $b$ and suppose that this set contains $m_i$ points of ramification index $e_i$, with the $e_i$ distinct and with $i$ running from $1$ to $r$, say. Then the \emph{ramification structure of $\varphi$ at $b$} is the set $R_b := \{ (e_1,m_1), \ldots, (e_r, m_r) \}$. The \emph{ramification structure of $\varphi$} is the ordered vector $R := (R_{b_i} : i=1,\ldots,n)$.
\end{definition}

\begin{remark}
  The ramification structure $R$ depends on $B$ and on the ordering of the points in $B$ --- even though this is not emphasized by our notation, the choice of $b_1,\ldots,b_n$ should always be clear from the context. Note furthermore that the definition above allows one to include the ramification structure at $b$ for points in the complement of the branch locus. In this case, the ramification structure at $b$ is $R_b = \{(1, \deg \varphi)\}$: all $\deg \varphi$ points in the fiber over $b$ have ramification index $1$.
\end{remark}

\begin{remark}
  We will connect ramification structures with the cycle type of certain permutations. We therefore agree to also write cycle types in the previous way: if the permutation $\sigma$ contains $m_i$ cycles of length $e_i$, with the $e_i$ distinct and with $i$ running from $1$ to $r$, say, then we write its cycle type as $\{ (e_1,m_1), \ldots, (e_r,m_r) \}$.
\end{remark}

\begin{example}
  Let $X$ be a smooth projective curve of genus $3$, $Y$ be an elliptic curve, and $\varphi : X \to Y$ be a morphism of degree 2. The Riemann-Hurwitz formula immediately implies that $\varphi$ is ramified at exactly 4 points, each with ramification index 2. If we take $B$ to be the branch locus of $\varphi$ (consisting of 4 points, ordered arbitrarily), then the ramification structure of $\varphi$ is $(\{(2,1)\},\{(2,1)\},\{(2,1)\},\{(2,1)\})$.
\end{example}

We now recall some basic facts about monodromy. Consider a morphism $\varphi : X \to Y$ between smooth projective curves over $\mathbb{C}$. Let $B=(b_1,\ldots,b_n)$ be a finite ordered subset of $Y$ which contains the branch locus of $\varphi$, and fix a base point $q \in Y \minus B$. Also fix loops $\gamma_1,\ldots,\gamma_n$, based at $q$, with the property that $\gamma_i$ is nontrivial in $\pi_1(Y \minus B, q)$ but trivial in $\pi_1(Y \minus (B \minus \{ b_i \}), q)$, and that winds precisely once in the counter-clockwise direction around $b_i$. We will call such a loop a \emph{small loop} based at $q$ around $b_i$. The classes $[\gamma_1], \ldots, [\gamma_n]$ then generate the fundamental group of $Y \minus B$. One can classify all maps $\varphi$ with branch locus contained in $B$ and of fixed degree in terms of representations of the fundamental group $\pi_1(Y \minus B,q)$. More precisely, we have

\begin{theorem}[{\cite[Proposition 4.9]{Miranda}}]\label{thm:BranchedCoversGeneral}
Let $Y$ be a compact Riemann surface, $B$ be a finite subset of $Y$, and let $q$ be a base point of $Y \minus B$. There is a bijection
\begin{equation*}
\left\{ \begin{array}{c}
\text{isomorphism classes of} \\
\text{holomorphic maps } \varphi : X \to Y \\
\text{of degree }d \\
\text{whose branch points} \\
\text{lie in }B
\end{array} \right\} \leftrightarrow \left\{
\begin{array}{c}
\text{group homomorphisms} \\
\rho : \pi_1(Y \minus B, q) \to S_d \\
\text{with transitive image}\\
\text{up to conjugacy in }S_d
\end{array}
\right\}
\end{equation*}
denoted by $\varphi_{\rho} \leftrightarrow \rho$ and $\varphi \leftrightarrow \rho_\varphi$.
If $\gamma_i$ is a small loop based at $q$ around $b_i \in B$, the ramification structure of $\varphi_\rho$ at $b_i$ is the cycle type of $\sigma_i:=\rho([\gamma_i])$.
\end{theorem}

As an immediate consequence of the previous theorem we have:
\begin{corollary}\label{cor:RamificationFromRepresentation}
With the same notation as in the theorem, the ramification structure of $\varphi_\rho: X \to Y$ is determined by the conjugacy classes in $S_d$ of $\rho([\gamma_i])$ for $i=1,\ldots,n$.
\end{corollary}

\begin{definition}
In the situation of the previous theorem, we will call the vector $\Sigma = (\sigma_1, \ldots, \sigma_n) = ( \rho([\gamma_1]), \ldots, \rho([\gamma_n]) )$ the \emph{monodromy datum} associated with $\varphi$.
\end{definition}

\begin{remark}\label{rmk:ModuliSpaceDimension}
  While the monodromy datum $\Sigma$ alone does not uniquely identify a map $\varphi: X \to Y$ (even up to isomorphism), because one also needs to specify the ordered set of points $( b_1,\ldots,b_n )$, any choice of such an ordered set will lead to a map $\varphi$ with the same ramification structure. Recall from \cite[\S 3]{MSSV02} that the dimension of the moduli space of covers of $\mathbb{P}^1$ branced over $n$ points has dimension $n-3$, so by letting the branch locus vary we get $(n-3)$-dimensional families of curves with fixed monodromy.
\end{remark}

We now specialize this discussion to the case $Y=\mathbb{P}^1$. The fundamental group of $\mathbb{P}^1 \minus B$ is generated by $[\gamma_1], \ldots, [\gamma_n]$, subject to the only relation $\prod_{i=1}^n [\gamma_i]=1$. Thus, given $\sigma_1,\ldots,\sigma_n \in S_d$ that satisfy $\prod_{i=1}^n \sigma_i=1$, we can define a homomorphism
\begin{equation*}
  \rho : \pi_1(\mathbb{P}^1 \minus B, q) \to S_d
\end{equation*}
by sending $[\gamma_i]$ to $\sigma_i$, and every homomorphism arises in this way for some $(\sigma_1,\ldots,\sigma_n)$. Thus we obtain the following special important case of Theorem \ref{thm:BranchedCoversGeneral}:

\begin{theorem}[{\cite[Corollary 4.10]{Miranda}}]\label{thm:BranchedCoversP1}
There is a bijective correspondence
\begin{equation*}
\left\{ \begin{array}{c}
\text{isomorphism classes of} \\
\text{holomorphic maps } \varphi : C \to \mathbb{P}^1 \\
\text{of degree }d \\
\text{whose branch points} \\
\text{lie in }B
\end{array} \right\} \leftrightarrow \left\{
\begin{array}{c}
\text{conjugacy classes of $n$-tuples} \\
(\sigma_1,\ldots,\sigma_n) \text{ of permutations in }S_d \\
\text{such that } \sigma_1\cdots\sigma_n=1 \\
\text{and the subgroup generated by the } \sigma_i \\
\text{is transitive}
\end{array}
\right\}
\end{equation*}
which enjoys the following additional property: the ramification structure at $b_i$ of the map $\varphi$ corresponding to $(\sigma_1,\ldots,\sigma_n)$ is the cycle type of $\sigma_i$.
\end{theorem}

\subsection{Galois closure of a morphism of curves}

Given a non-constant morphism $\varphi : X \to Y$ of smooth projective curves over $\mathbb{C}$, it makes sense to consider the corresponding (finite, separable) field extension $\varphi^*\mathbb{C}(Y) \subseteq \mathbb{C}(X)$. As with any such extension, we can then consider the Galois closure of $\mathbb{C}(X)$ over $\varphi^*\mathbb{C}(Y)$, which by the equivalence between smooth projective curves over $\mathbb{C}$ and extensions of $\mathbb{C}$ of transcendence degree 1 corresponds to some curve $\tilde{C}$ equipped with a canonical morphism $\tilde{C} \to X$. We call $\tilde{C}$ (equipped with its maps $\tilde{C} \to X \to Y$) the Galois closure of $X \to Y$, and we say that $\tilde{C}/Y$ has Galois group $G$ if this is true for the corresponding extension of function fields. There is a natural action of $G$ on $\tilde{C}$, and for a subgroup $H$ of $G$ we write $\tilde{C}/H$ for the curve corresponding to the subfield of $\mathbb{C}(\tilde{C})$ fixed by $H$.

We now recall a description of the Galois closure in terms of the monodromy datum. Suppose the map $\varphi : X \to Y$ corresponds, as in Theorem \ref{thm:BranchedCoversGeneral}, to $B=(b_1,\ldots,b_n)$ and to the representation $\rho$. As in the statement of the theorem, let $\gamma_i$ be a small loop based at $q$ around $b_i$. Finally, let $\sigma_i= \rho([\gamma_i])$. Then we have the following description of the Galois closure of $\varphi$:

\begin{theorem}\label{thm:GaloisClosure} Let $\tilde{\varphi} : \tilde{C} \to X \to Y$ be the Galois closure of $\varphi : X \to Y$. Then:
\begin{enumerate}
\item the Galois group of $\tilde{C}/Y$ is the subgroup $G$ of $S_d$ generated by the $\sigma_i$, and the degree of $\tilde{\varphi}$ is $|G|$;
\item the branch locus of $\tilde{\varphi}$ is contained in $B$;
\item the corresponding representation $\rho_{\tilde{\varphi}}$ is obtained as follows: identifying $S_{|G|}$ with the group of permutations of the elements of $G$, the class $[\gamma_i]$ is sent to the permutation of $G$ induced by left-multiplication by $\sigma_i$.
\end{enumerate}
\end{theorem}

This is all explained in \cite{bertin-algstacks}, which, however, does not contain a separate statement that comprises all three items above. We therefore include a short proof with more detailed references:

\begin{proof}
  Part (i) is in \cite[§4.3.1]{bertin-algstacks}. Part (ii) follows from the equivalence between curves and function fields: a point $b \in Y$ is a branch point for $\varphi : X \to Y$ precisely when the corresponding place of $\mathbb{C}(Y)$ ramifies in $\mathbb{C}(X)$. Moreover, because a compositum of unramified extensions of local fields is unramified \cite[II.7.3]{neukirch-ant}, the branched places of the extension $\mathbb{C}(X)/\varphi^*\mathbb{C}(Y)$ coincide with those of its Galois closure. Finally, (iii) is part of the theory of $G$-sets \cite[Chapter V]{massey-at}, \cite[Chapter 1]{lenstra-galois}. More generally, if $H$ is any subgroup of $G$, then the fiber of $\tilde{C}/H \to \tilde{C}/G=Y$ is identified with $G/H$, and the monodromy action is the natural multiplication action of $G$ on $G/H$. Applying this to $H=\{1\}$ yields the result.
\end{proof}

\begin{remark}
If $\varphi : X \to Y$ corresponds to the monodromy datum $\Sigma=(\sigma_1,\ldots,\sigma_n)$, we will denote by $\tilde{\sigma}_i$ the permutation $\rho_{\tilde{\varphi}}([\gamma_i])$ and by
$\tilde{\Sigma}$ the vector $(\tilde{\sigma}_1,\ldots,\tilde{\sigma}_n)$.
\end{remark}

\subsection{Statement of the problem}

We begin by describing the objects of interest:
\begin{definition}\label{def:DiagramOfGivenType}
Consider a 5-tuple $(g_X, g_Y, d_X, d_Y, R)$, where $g_X, g_Y$ are non-negative integers, $d_X, d_Y$ are positive integers, and $R=(R_1,\ldots,R_n)$ is a ramification structure, that is, a collection of pairs $R_i=(e_i, m_i)$ of positive integers. A \emph{diagram of type $(g_X, g_Y, d_X, d_Y, R)$} is a diagram of maps of smooth projective curves \begin{equation}\label{eq:Diagram1}
\xymatrix{
Z \ar[r] &
X \ar[r]^{\pi_{X/Y}} &
Y \ar[r]^{\pi_{Y/\mathbb{P}^1}} &
\mathbb{P}^1
}
\end{equation}
that satisfies the following properties:
\begin{enumerate}
\item the genera of $X$ and $Y$ are $g_X, g_Y$ respectively;
\item $\pi_{X/Y}$ is of degree $d_X$ and $\pi_{Y/\mathbb{P}^1}$ is of degree $d_Y$;
\item the branch locus of $X \to \mathbb{P}^1$ is contained in an ordered set $B= (b_1,\ldots,b_n)$ with $n$ elements;
\item the ramification structure of $X \to \mathbb{P}^1$, computed with respect to $B$, is equal to $R$;
\item $Z \to \mathbb{P}^1$ is the Galois closure of $X \to \mathbb{P}^1$.
\end{enumerate}
\end{definition}

\begin{remark}
Note that the number of branching points of $X \to \mathbb{P}^1$ is precisely $n$ if and only if none of the $R_i$ is equal to $\{(1,d_Xd_Y)\}$. Indeed, this ramification structure denotes a point whose fiber contains $d_Xd_Y$ points, none of which is ramified.
\end{remark}

The map $X \to \mathbb{P}^1$ corresponds to a monodromy datum $\Sigma$ as in Theorem \ref{thm:BranchedCoversP1}.
Let $G$ be the Galois group of $Z/\mathbb{P}^1$, and let $\tilde{\Sigma}$ be the corresponding monodromy datum. On the function fields side we have corresponding inclusions $\mathbb{C}(\mathbb{P}^1) \subseteq \mathbb{C}(Y) \subseteq \mathbb{C}(X) \subseteq \mathbb{C}(Z)$, and by Galois correspondence we obtain subgroups $H_X, H_Y$ of $G$ with the property that $Z/H_X=X$ and $Z/H_Y=Y$. In what follows we will be interested in 4-tuples $(G,H_X,H_Y,\Sigma)$ that arise from this construction.

\begin{remark}\label{rmk:Stab1}
  Let $d = d_X d_Y$. The construction of $Z \to \P^1$ as the Galois closure of $X \to \P^1$ amounts to fixing a distinguished embedding of the Galois group $G$ into $S_d$, for which $H_X$ is conjugate to the stabilizer of $1$. This leads to a corresponding notion of isomorphism, which is that of simultaneous conjugation of the 4-tuple $(G,H_X,H_Y,\Sigma)$ in $S_d$. That is, if $g$ is any element of $S_d$, and $\Sigma=(\sigma_1,\ldots,\sigma_n)$, then we write $g\Sigma g^{-1}$ for the vector $(g\sigma_ig^{-1})_{i=1,\ldots,n}$ and say that the 4-tuples $(G,H_X,H_Y,\Sigma)$ and $(gGg^{-1},gH_Xg^{-1},gH_Yg^{-1},g\Sigma g^{-1})$ are isomorphic.
\end{remark}

The problem we will solve is the following. Fix a 5-tuple $(g_X,g_Y,d_X,d_Y,R)$ as in Definition \ref{def:DiagramOfGivenType} and let $X \to Y \to \mathbb{P}^1$ be a diagram of type $(g_X, g_Y, d_X, d_Y, R)$. Let $(G, H_X, H_Y, \Sigma)$ be the corresponding 4-tuple constructed above.
Up to isomorphism there are only finitely many possibilities for $(G,H_X,H_Y,\Sigma)$, and our first algorithmic task is the following:

\begin{problem}\label{problem:FirstTask} Given $(g_X,g_Y,d_X,d_Y,R)$ as in Definition \ref{def:DiagramOfGivenType}, output a list $\mathcal{L}(g_X,g_Y,d_X,d_Y,R)$ of all isomorphism classes of 4-tuples $(G,H_X,H_Y,\Sigma)$ that can be obtained from a diagram $X \to Y \to \mathbb{P}^1$ of type $(g_X,g_Y,d_X,d_Y,R)$.
\end{problem}

Note that a list $\mathcal{L}(g_X,g_Y,d_X,d_Y,R)$ as in the statement of Problem \ref{problem:FirstTask} gives a complete set of representatives of isomorphism classes of diagrams of type $(g_X,g_Y,d_X,d_Y,R)$, in the following precise sense. Suppose we have a diagram of type $(g_X,g_Y,d_X,d_Y,R)$: then one of the 4-tuples $(G,H_X,H_Y,\Sigma)$ in $\mathcal{L}(g_X,g_Y,d_X,d_Y,R)$ enjoys the following properties. Consider the unique (up to isomorphism) cover $X'$ of $\mathbb{P}^1$ of degree $d_Xd_Y$, branched at most over $B$, and corresponding to the monodromy datum $\Sigma$. Also let $Z'$ be the Galois closure of $X' \to \mathbb{P}^1$. The following holds:
\begin{enumerate}
\item we have $\operatorname{Aut}(Z/\mathbb{P}^1) \cong \operatorname{Aut}(Z'/\mathbb{P}^1) \cong G$;
\item there is a canonical identification $X' = Z'/H_X$;
\item the map $Z' \to \mathbb{P}^1$ is isomorphic to $Z \to \mathbb{P}^1$ as a $G$-cover;
\item the $G$-isomorphism $Z' \cong Z$ can be chosen in such a way that $Z'/H_X$ is carried to $X$ and $Z'/H_Y$ is carried to $Y$;
\item in particular, the monodromy datum attached to $X \to \mathbb{P}^1$ is equivalent to $\Sigma$ (up to conjugacy in the symmetric group).
\end{enumerate}

\begin{remark}\label{rmk:EveryDiagramArisesFromAType}
Informally, this means that a diagram of type $(g_X,g_Y,d_X,d_Y,R)$ arises from one of the monodromy data $\Sigma$ found in $\mathcal{L}(g_X,g_Y,d_X,d_Y,R)$, the only information missing being the ordered set of branch points.
\end{remark}

In addition, for each $(G,H_X,H_Y,\Sigma)$ we would like to extract some additional information:

\begin{problem}\label{problem:SecondTask}
 Given $(G,H_X,H_Y,\Sigma)$, determine:
\begin{enumerate}
  \item for every pair of subgroups $H_1 < H_2 < G$, the degree and ramification structure of the corresponding map $Z/H_1 \to Z/H_2$;
  \item for every subgroup $H$ of $G$, the genus of the curve $Z/H$;
  \item the action of $G$ on the vector space $H^0(Z,\Omega^1_Z)$ induced by the natural action of $G$ on $Z$;
  \item for every pair of subgroups $H_1, H_2$ of $G$, the dimension of the image of the map on Jacobians $\Jac(Z/H_1) \to \Jac(Z/H_2)$ induced by the correspondence
    \begin{equation}\label{eq:Correspondence}
      \xymatrix{
      & Z \ar[ld]_{\pi_1} \ar[rd]^{\pi_2} & \\
      Z/H_1 &  & Z/H_2
      }
    \end{equation}
\end{enumerate}
\end{problem}

\subsection{Theory}\label{subsec:AlgoTheory}
We now review the theoretical tools necessary to solve Problem \ref{problem:SecondTask}. Our input data is a 4-tuple $(G,H_X,H_Y,\Sigma)$, corresponding to a diagram of type $(g_X,g_Y,d_X,d_Y,R)$.

\subsubsection{Degree and ramification structure of $Z/H_1 \to Z/H_2$.} Galois theory immediately shows that the degree of the natural projection $Z/H_1 \to Z/H_2$ is equal to $[H_2 : H_1]$.

As for the ramification structure, we begin with the special case $H_2=G$ and $H_1$ arbitrary. The quotient $Z/H_2$ is therefore equal to $\mathbb{P}^1$, the curve $Z/H_1$ is a branched cover of it, and we may rely on Theorem \ref{thm:BranchedCoversP1} to describe its ramification. In fact, the theorem shows that it suffices to understand the monodromy representation corresponding to $\pi : Z/H_1 \to \mathbb{P}^1$. Let $B$ be the set (containing the branch locus) that defines the cover $Z \to \mathbb{P}^1$ and let $B_Z$ (resp.\ $B_{Z/H_1}$) be the inverse images of $B$ in $Z$ (resp.\ $Z/H_1$). Let $Z^0 := Z\minus B_Z$ and observe that $Z^0/H_1$ coincides with $Z/H_1 \minus B_{Z/H_1}$.
We have a diagram of étale maps
\begin{equation*}
  Z^0 \to Z^0/H_1 \to \mathbb{P}^1 \minus B
\end{equation*}
which we may study via the usual topological interpretation of coverings as $\pi_1$-sets. In particular, fixing a base point $q \in \mathbb{P}^1 \minus B$, one may identify the fiber of $Z^0$ over $q$ with $G$ and the fiber of $(Z/H_1)^0$ with $G/H$. In this language, the monodromy datum $\Sigma$ gives rise to a representation
\begin{equation*}
  \rho: \pi_1 ( \mathbb{P}^1 \minus B, q ) \to G:
\end{equation*}
the $\pi_1$-structure of $G$ is then $\gamma \cdot g := \rho(\gamma)g$ for $\gamma \in \pi_1( \mathbb{P}^1 \minus B, q )$.
The $\pi_1$-set corresponding to $Z^0/H_1$ is then the set $G/H_1$, equipped with the action $\gamma \cdot gH_1 := \rho(\gamma)gH_1$. We can now translate back to the language of monodromy datum: for each $i=1,\ldots,n$ we have a permutation of the set $G/H_1$, defined by left-multiplication by the element $\sigma_i$. We may then use Theorem \ref{thm:BranchedCoversP1} to describe the ramification structure of $Z/H_1 \to \mathbb{P}^1$, and we obtain:

\begin{proposition}\label{prop:RamificationOverP1}
  Let $H_1$ be a subgroup of $G$. The branched cover $Z/H_1 \to \mathbb{P}^1$ is ramified at most over the points in $B=(b_1,\ldots,b_n)$. The ramification over $b_i$ can be determined as follows: consider the left multiplication of $\sigma_i$ on the quotient set $G/H_1$. This induces a permutation of $G/H_1$, with cycle type $((e_1, m_1),\ldots,(e_k, m_k))$. Then for all $1 \leq j\leq k$, the fiber over $b_i$ contains exactly $m_j$ points with multiplicity $e_j$, and no other points beyond these.
\end{proposition}

\begin{remark}\label{rmk:fiberOverBranchPoint}
The last statement in the previous proposition shows that the fiber of $Z/H_1 \to \mathbb{P}^1$ over a point $b_i \in B$ is in natural bijection with the double coset space $ \langle \sigma_i \rangle \backslash G/H_1 $.
\end{remark}

We will also need the following straightforward generalization of Proposition \ref{prop:RamificationOverP1}, which follows upon replacing Theorem \ref{thm:BranchedCoversP1} with Theorem \ref{thm:BranchedCoversGeneral}:
\begin{proposition}\label{prop:RamificationOverGaloisQuotient}
  Let $Z$ be a smooth projective curve over $\mathbb{C}$ with an action of a group $G$, and let $H$ be a subgroup of $G$. Let $B=(b_1,\ldots,b_n) \subseteq Z/H$ be a finite ordered subset containing the branch locus of $\pi_H : Z \to Z/H$ and let $\rho$ be the corresponding representation $\pi_1(Z/G \minus B, q) \to S_d$ as in Theorem \ref{thm:BranchedCoversGeneral}. Finally, let $\gamma_i$ be small loops based at $q$ around each $b_i$ and let $\sigma_i=\rho([\gamma_i])$. Recall from Theorem \ref{thm:GaloisClosure} that $G$ is identified with the subgroup of $S_d$ generated by the $\sigma_i$. The ramification of $\pi_H$ over $b_i$ can be determined as follows. Consider the left multiplication by $\sigma_i$ on the quotient set $G/H$: it induces a permutation of the set $G/H$, with cycle type $(e_1,\ldots,e_k)$. The fiber over $b_i$ consists of $k$ points, of multiplicities $e_1,\ldots,e_k$.
\end{proposition}

Second, we treat the case of a Galois cover $\pi_H : Z \to Z/H$. This is discussed for example in \cite[Proposition 2.2.2]{bertin-romagny-champshurwitz} and in \cite[§4]{MSSV02}, so we only recall the result. Let as before $Z^0$ be the complement in $Z$ of the inverse image of $B$, and observe that we have a tower of topological covers
\begin{equation*}
Z^0 \xrightarrow{\pi_H} Z^0/H \xrightarrow{\varphi} \mathbb{P}^1 \minus B:
\end{equation*}
in particular, $\pi_H$ is unramified outside of the inverse image of $B$ in $Z/H$. Thus the branch locus of $\pi_H$ is contained in $\varphi^{-1}(B)$, and we have a description of this set by the special case we treated above: by Remark \ref{rmk:fiberOverBranchPoint}, the set $\varphi^{-1}(B)$ can be parametrized by pairs $(b_i,\langle\sigma_i\rangle g H)$, where the second coordinate is an element in the double coset space $\langle \sigma_i \rangle \backslash G / H$.
The monodromy operator given by a small loop around the point corresponding to $\langle \sigma_i \rangle g H$ is obtained as follows: letting $m_{g,i}$ be the smallest positive integer for which $g^{-1} \sigma_i^{m_{g,i}} g \in H$, the monodromy operator is precisely $g^{-1} \sigma_i^{m_{g,i}} g$.

The case of a general intermediate cover $\pi : Z/H_1 \to Z/H_2$ follows upon combining the previous two special cases: we first obtain the monodromy datum of $Z \to Z/H_2$ in the way just described, and then deduce that of $Z/H_1 \to Z/H_2$ by applying Proposition \ref{prop:RamificationOverGaloisQuotient}. This leads to the following algorithmic procedure to express the monodromy of $Z/H_1 \to Z/H_2$ in terms of $(G,H_1,H_2,\Sigma)$:

\begin{algorithm}\label{algo:ComputeRamification}
Input: $(G,H_1,H_2,\Sigma)$ with $H_1 < H_2$ and $\Sigma=(\sigma_1,\ldots,\sigma_n)$.

Output: the ramification structure of $Z/H_1 \to Z/H_2$.

Procedure:
\begin{enumerate}
\item for every $i=1,\ldots,n$:
\begin{enumerate}
\item compute representatives $\langle \sigma_i\rangle g_{ij} H_2$ for the double coset space $\langle \sigma_i \rangle \backslash G/ H_2$.
\item for each $g_{ij}$:
\begin{enumerate}
\item let $m_{ij}$ be the least positive integer for which $g_{ij}^{-1} \sigma_i^{m_{ij}} g_{ij}$ lies in $H_2$. Set $\sigma_{ij} = g_{ij}^{-1} \sigma_i^{m_{ij}} g_{ij}$.
\item compute the permutation of $G/H_1$ induced by left multiplication by $\sigma_{ij}$. Let $R_{ij}$ be the cycle type of this permutation.
\end{enumerate}
\end{enumerate}
\item The ramification structure of $Z/H_1 \to Z/H_2$ is the vector $(R_{ij} \bigm\vert i=1,\ldots,n, \; \langle \sigma_i \rangle g_{ij}H_2 \in \langle \sigma_i \rangle \backslash G / H_2  )$.
\end{enumerate}

\end{algorithm}

\subsubsection{The genera of the curves $Z/H$.}\label{sec:genus} By the previous paragraph we know how to read off our data the ramification structure of the map $\varphi : Z/H \to Z/G=\mathbb{P}^1$. In particular, we know the multiplicity of each ramification point $y_i \in Z/H$, and since we also know $\deg \varphi = [G:H]$ we can simply apply the Riemann-Hurwitz formula to obtain
\begin{equation}\label{eq:genform}
g(Z/H) = \frac{1}{2} \left( 2-2[G:H] + \sum_{y \in Z/H} (e(y)-1) \right).
\end{equation}

\subsubsection{$G$-module structure of $H^0(Z,\Omega^1_Z)$.}\label{subsec:GModuleStructure}
To extract this information from $(G,H_X,H_Y,\Sigma)$ we use a beautiful theorem due to Chevalley and Weil \cite{chevalleyweil, weil2} that we now recall.

We need some preliminary notation. Let $B=(b_1,\ldots,b_n)$ be the ordered branch locus of $Z \to \mathbb{P}^1$ and consider one of the branch points $b_i \in B$. As part of our data we have access to a permutation $\sigma_i \in G$ corresponding to the branch point $b_i$. Let $e_i$ be the order of the permutation $\sigma_i$, or equivalently (by Theorem \ref{thm:GaloisClosure}) the ramification index of any point of $Z$ lying over $b_i$. Fix once and for all a primitive $|G|$-th root of unity $\zeta \in \mathbb{C}$, and, for any divisor $e$ of $|G|$, denote by $\zeta_e$ the complex number $\zeta^{|G|/e}$.

Observe that $V:=H^0(Z,\Omega^1_Z)$ is a $\mathbb{C}[G]$-module in a natural way, and it is automatically semisimple since $\mathbb{C}$ is of characteristic $0$. In order to describe the $\mathbb{C}[G]$-module structure of $V$, therefore, it suffices to give the multiplicity of each irreducible representation of $G$ in $V$.
For a fixed linear representation $\tau$ of $G$, denote by $N_{i,\alpha} = N_{i,\alpha} (\tau)$ the multiplicity of $\zeta_{e_i}^{\alpha}$ as eigenvalue of $\tau(\tilde{\sigma}_i)$, where $\tilde{\sigma}_i$ is the monodromy operator corresponding to the cover $Z \to \mathbb{P}^1$ and the point $b_i$. With this notation, and in the special case of covers of $\mathbb{P}^1$, the Chevalley-Weil formula reads as follows:

\begin{theorem}[Chevalley-Weil]\label{thm:Chevalley-Weil}
Let $\varphi : Z \to \mathbb{P}^1$ be a branched Galois cover of smooth projective complex algebraic curves, let $B$ be its branch locus, and let $G$ be the corresponding Galois group. Let $\tau_\chi$ be an irreducible linear complex representation of $G$ with character $\chi : G \to \mathbb{C}$ and define $e_i$ and $N_{i,\alpha}:=N_{i,\alpha}(\tau_\chi)$ as above. The multiplicity $\nu_\chi$ of $\tau_\chi$ in the $G$-representation $H^0(Z, \Omega^1_{Z})$ is given by
  \begin{equation}
    \nu_\chi = -d_\chi+\sum_{i=1}^p\sum_{\alpha=0}^{e_i-1} N_{i,\alpha} \left\langle -\frac{\alpha}{e_i} \right\rangle + \sigma,
  \end{equation}
where $d_{\chi}=\chi(1)$ is the dimension of $\tau_\chi$ and
\begin{equation*}
\sigma = \begin{cases}
1 \text{ if } \chi \text{ is the trivial character}\\
0 \text{ otherwise}.
\end{cases}
\end{equation*}
Finally, $\left\langle x \right\rangle=x-\lfloor x \rfloor \in [0,1)$ denotes the fractional part of the real number $x$.
\end{theorem}

Note that the multiplicity $\nu_\chi$ is determined by $(G, H_X, H_Y, \Sigma)$ and $\chi$: we have already observed that $e_i$ is the order of $\sigma_i$, and explained how to obtain the monodromy datum $\tilde{\Sigma}$ (see Theorem \ref{thm:GaloisClosure}). Finally, the number $N_{i,\alpha}$ is the multiplicity of $\zeta_{e_i}$ (which is a known complex number) as an eigenvalue of $\tau_\chi(\tilde{\sigma}_i)$, and a basic result in representation theory shows that $\tau_{\chi}$ is in turn determined by $\chi$, so that $\nu_\chi$ is indeed determined by $(G, H_X, H_Y, \Sigma)$. The upshot of this discussion is that we have an isomorphism of $\mathbb{C}[G]$-modules $ V \cong \bigoplus_{\chi} \tau_\chi^{\oplus \nu_\chi}$, where all the objects on the right hand side are determined by $(G,H_X,H_Y,\Sigma)$ as desired.

\subsubsection{The maps $\Jac(Z/H_1) \to \Jac(Z/H_2)$.}\label{sec:MapsJacobians}
Our last objective is to understand the image of the maps $\Jac(Z/H_1) \to \Jac(Z/H_2)$ induced by the correspondence \eqref{eq:Correspondence}.
Note that the complex vector space $V=H^0(Z,\Omega^1_Z)$ provides the natural analytic uniformization of $\Jac(Z)$, and that the maps $\Jac(Z/H_i) \to \Jac(Z)$ are induced by the pullback $\pi_i^* : H^0(Z/H_i,\Omega^1_{Z/H_i}) \to H^0(Z,\Omega^1_Z)$. Thus it suffices to study the map $\pi_{2*} \circ \pi_1^* : H^0(Z/H_1, \Omega^1_{Z/H_1}) \to H^0(Z/H_2, \Omega^1_{Z/H_2})$. Note that the pushforward $\pi_{2*}$ makes sense since $Z \to Z/H_2$ is a finite (albeit ramified) cover.
We will need a result from representation theory:
\begin{theorem}[{\cite[Théorème 2.6.8]{serre-repgroupesfinis}}]\label{thm:RepTheory}
  Let $\tau : G \to \operatorname{GL}(V)$ be a finite-dimensional linear complex representation of the finite group $G$ and let $H$ be a subgroup of $G$. Define $p_H := \frac{1}{\#H} \sum_{h \in H} \tau(h) \in \operatorname{End}(V)$. Then $p_H$ is a projector, that is, $p_H^2=p_H$, and its image is precisely the $H$-invariant subspace of $V$.
\end{theorem}
\begin{remark}
We will only work with the representation of $G$ afforded by $V=H^0(Z,\Omega^1_Z)$, so, for the sake of simplicity, given a subgroup $H$ of $G$ we will simply write $p_H = \frac{1}{\#H} \sum_{h \in H} h$, omitting the representation $\tau$.
\end{remark}

In order to connect the maps $\pi_{i*}$ and $\pi_i^*$ with representation theory we will make use of the following result:
\begin{theorem}\label{thm:PusforwardPullbackDifferentials}
  Let $H$ be a subgroup of $G$ and let $\pi : Z \to Z/H$ be the corresponding quotient map. Then:
  \begin{enumerate}
    \item $\pi^* : H^0(Z/H, \Omega^1_{Z/H}) \to H^0(Z,\Omega^1_Z)$ is injective, and its image is the $H$-invariant subspace of $H^0(Z,\Omega^1_Z)$;
    \item $\pi^*\pi_* : H^0(Z,\Omega^1_Z) \to H^0(Z,\Omega^1_Z)$ coincides with the operator $\#H \cdot p_{H}$.
  \end{enumerate}
\end{theorem}

Part (i) is well-known; we include a short proof of (ii):
\begin{proof}[Proof of (ii)]
Since a curve and its Jacobian share the same space of regular differentials, it suffices to prove the same statement with $Z, Z/H$ replaced by their Jacobians. We prove the stronger statement that the required relation is true for the divisor groups. Let $D= \sum_i n_i P_i$ be a divisor on $Z$. By definition, $\pi_*D = \sum_{i} n_i \pi(P_i)$. Since the fiber over $\pi(P_i)$ is the divisor given by the sum of all the points that map to $\pi(P_i)$, namely, $\sum_{h \in H} h \cdot P_i$, we obtain $\pi^*\pi_*D = \sum_{i} n_i \sum_{h \in H} h \cdot P_i = \#H \cdot p_H(D)$ as desired.
\end{proof}

We wish to determine the dimension of $\operatorname{Im} \left( \Jac(Z/H_1) \to \Jac(Z/H_2) \right)$, or equivalently the dimension of $\pi_{2*} \circ \pi_1^*(H^0(Z/H_1, \Omega^1_{Z/H_1}))$. Since $\pi_{2}^*$ is injective, we may as well study the dimension of $\pi_{2}^* \circ \pi_{2*} \circ \pi_1^*(H^0(Z/H_1, \Omega^1_{Z/H_1}))$. By Theorem \ref{thm:PusforwardPullbackDifferentials}, $\pi_1^*(H^0(Z/H_1, \Omega^1_{Z/H_1}))$ is precisely the $H_1$-invariant subspace of $V$, hence (by Theorem \ref{thm:RepTheory}) it is the image of $p_{H_1}$. We may easily identify this subspace, because we have already shown how to write down a representation isomorphic to $V$. It follows that \begin{equation*}
\pi_{2}^* \circ \pi_{2*} \circ \pi_1^*(H^0(Z/H_1, \Omega^1_{Z/H_1})) = \#H_2 \cdot p_{H_2} \pi_1^* \left( H^0(Z/H_1, \Omega^1_{Z/H_1}) \right) = \#H_2 \cdot p_{H_2} \cdot p_{H_1} (V)
\end{equation*}
has dimension equal to the rank of the operator $p_{H_2} \cdot p_{H_1}$.
We have obtained:
\begin{proposition}\label{prop:projector}
  The dimension of the image of the map $\Jac(Z/H_1) \to \Jac(Z/H_2)$ induced by the correspondence $Z$ is equal to the rank of
  \begin{equation}
    \left( \sum_{h_2 \in H_2} h_2 \right)\left( \sum_{h_1 \in H_1} h_1 \right) : V \to V.
  \end{equation}
\end{proposition}
Since we have already shown that the action of $G$ on $V$ is completely determined by the monodromy datum $\Sigma$, this allows us to express the dimension of $\operatorname{Im} \left( \Jac(Z/H_1) \to \Jac(Z/H_2) \right)$ in terms of $(G, H_X, H_Y, \Sigma)$.

Note furthermore that the same machinery allows us to also answer a slightly different question: for example, in our application we consider diagrams of curves of the form
\begin{equation*}
\xymatrix{
& Z \ar[dl]_{\pi_2} \ar[dr]^{\pi_1} \\
X \ar[d]_{\pi_3} & & C \\
Y
}
\end{equation*}
and we need to understand whether or not the image of the map $\Jac(C) \to \Jac(X)$ induced by the correspondence $Z$ intersects the image of the map $\Jac(Y) \to \Jac(X)$ induced by pulling back divisors from $Y$ to $X$.
Passing to analytic uniformizations, the question is whether the subspaces $\pi_{2*} \circ \pi_1^* (H^0(C, \Omega^1_c))$ and $\pi_3^* (H^0(Y, \Omega^1_Y))$ of $H^0(X,\Omega^1_X)$ intersect nontrivially. However, since $\pi_1^*, \pi_2^*, \pi_3^*$ are all injective, it suffices to know whether
\begin{equation*}
\pi_2^* \circ \pi_{2*} \circ \pi_1^* (H^0(C, \Omega^1_C)) \text{ and } \pi_2^* \circ \pi_3^* (H^0(Y, \Omega^1_Y)) = (\pi_3 \circ \pi_2)^*(H^0(Y, \Omega^1_Y))
\end{equation*}
intersect nontrivially inside $V$. Proceeding as above, and letting $H_X, H_Y, H_C$ be the subgroups of $G$ corresponding via Galois theory to $X, Y, C$ respectively, we conclude that the image of the map $\Jac(C) \to \Jac(X)$ induced by $Z$ intersects the image of $\Jac(Y) \to \Jac(X)$ if and only if the operator $p_{H_Y} p_{H_C}$ is nonzero.

\subsubsection{Conclusion.} Putting together the results of the previous paragraphs we obtain:
\begin{proposition}\label{prop:Theory}
  Let $Z \to \mathbb{P}^1$ be a Galois branched cover with group $G$. The monodromy datum of $Z \to \mathbb{P}^1$ determines the following (in an effectively computable way): For every subgroup $H < G$, the genus of $Z/H$, and for every pair of subgroups $H_1, H_2$ of $G$, the dimension of the image of the induced map $\Jac(Z/H_1) \to \Jac(Z/H_2)$.
\end{proposition}

%% file: implementation.tex

\section{Implementation} \label{sec:implem}

We now turn to details and optimizations concerning the implementation of Problems \ref{problem:FirstTask} and \ref{problem:SecondTask} in practice. Since our solution to Problem \ref{problem:FirstTask} actually relies on being able to handle Part (i) of Problem \ref{problem:SecondTask}, we begin with the latter.

\subsection{Solving Problem \ref{problem:SecondTask}}

This is a direct application of the theory explained in Section \ref{sec:AlgoTheory}:

\begin{algorithm}\label{algo:SecondTask}
  Input: a 4-tuple $(G,H_X,H_Y,\Sigma)$ as in Problem \ref{problem:SecondTask}.

  Output: the structure of $H^0(Z,\Omega^1_Z)$ as a $G$-representation; for each pair of subgroups $H_1 < H_2 < G$, the ramification structure of $Z/H_1 \to Z/H_2$, the genus of $Z/H_1$, and the dimension of the image of the map $\Jac(Z/H_1) \to \Jac(Z/H_2)$ induced by $Z$ as in \eqref{eq:Correspondence}.

  Procedure:
  \begin{enumerate}
    \item Compute the genera of $Z \to Z/H_1$ and the intermediate ramification of $Z/H_1 \to Z/H_2$ using \eqref{eq:genform} and Algorithm \ref{algo:ComputeRamification}.
    \item Compute the $G$-module structure of $H^0(Z,\Omega^1_Z)$ as explained in Section \ref{subsec:GModuleStructure}.
    \item Compute the dimension of $\im (\Jac(Z/H_1) \to \Jac(Z/H_2))$ using Theorem \ref{thm:Chevalley-Weil} and Proposition \ref{prop:Theory}.
  \end{enumerate}
\end{algorithm}

The theory behind each of these steps has been laid out in Section \ref{subsec:AlgoTheory}. These computations do not require the computation of curve equations and depend only on the specified ramification structure $\Sigma$ over the branch locus $B$ of $X \to \mathbb{P}^1$, and not on this branch locus itself. This independence of $B$ implies that all our calculations may be performed abstractly, and will be valid for any choice of $B$. This means that we actually consider families of examples of dimension $-3+r$, where $r$ is the number of branch points of $X \to \mathbb{P}^1$, see Remark \ref{rmk:ModuliSpaceDimension}.

Finally, note that as long as the degree of the composed map $X \to \P^1$ is small, the computations involved in \ref{algo:SecondTask}(i), which are described in Algorithm \ref{algo:ComputeRamification} and Section \ref{sec:genus}, take place in a symmetric group on a small set, and therefore terminate quickly. We discuss speedups for Parts (ii) and (iii) in the next section.

\subsection{Solving Problem \ref{problem:FirstTask}}

Let $(g_X, g_Y, d_X, d_Y, R)$ be given. We want to find the corresponding tuples $(G, H_X, H_Y, \Sigma)$. Recall that for $d = d_X d_Y$, the group $G$ is the subgroup of $S_d$ generated by the monodromy $\Sigma$. Moreover, we want that the corresponding Galois cover is the Galois closure of the cover of degree $d$ corresponding to $H_X$. We can ensure this by fixing an embedding of $G$ into $S_d$ and letting $H_X$ be the stabilizer of $1$. In other words, given $G$, determining the possible pairs $(G, H_X)$ comes down to realizing $G$ as a conjugacy class of subgroups of $S_d$. Moreover, when the group $G$ is not specified, we can find all pairs $(G, H_X)$ up to equivalence by running through the conjugacy classes of subgroups of $S_d$. For this latter problem, efficient algorithms exist when $d$ is small.

\begin{remark}
  Note that at the very least $G$ has to act transitively to correspond to a connected cover of $\P^1$. Moreover, we may restrict to subgroups $H_X$ with the property that $H_X$ has a normal subgroup of index at most $(d_Y - 1)!$, since only such subgroups can give rise to a diagram \eqref{eq:Diagram1} with the requested properties. (Indeed, the core of $H_X$ in $H_Y$ is a normal subgroup of $H_Y$ that is contained in $H_X$ and that is of index at most $d_Y!$ in $H_Y$.)
\end{remark}

It now remains to find all the possible extensions of a given pair $(G, H_X)$ to quadruples $(G, H_X, H_Y, \Sigma)$. Once $(G, H_X)$, or alternatively (by the above) an embedding of $G$ into $S_d$, is given, the remaining isomorphisms on the level of covers translate into conjugation by the normalizer $N_G$ of $G$ in $S_d$. We accordingly determine the subgroups $K_Y$ of $G$ of index $d_Y$ up to conjugacy by $N_G$. Having found this, we find representatives for triples $(G, H_X, H_Y)$ as follows:

\begin{proposition}\label{prop:doublecoset}
  The simultaneous $N_G$-conjugacy classes of triples $(G, K_X, K_Y)$ such that $K_X$ (resp.\ $K_Y)$ is $N_G$-conjugate to a given subgroup $H_X$ (resp.\ $H_Y$) of $G$ are in bijection with the double coset space $N_Y \backslash N_G / N_X$. Here $N_Y$ (resp.\ $N_X$) is the normalizer of $H_Y$ (resp.\ $H_X$) in $N_G$, and to a double coset $N_Y g N_X$ there corresponds the triple $(G, H_X, g^{-1} H_Y g)$.
\end{proposition}

\begin{proof}
  The indicated map is well-defined, and it is surjective since after conjugating by a suitable element of $N_G$ if necessary we may assume that $K_X = H_X$. Conversely, if two pairs $(G, H_X, n_1^{-1} H_Y n_1)$ and $(G, H_X, n_2^{-1} H_Y n_2)$ are simultaneously $N_G$-conjugate, then $n_1^{-1} H_Y n_1 = g n_2^{-1} H_Y n_2 g$ for some element $g$ of $N_X$, which implies that $n_2 g = h n_1$ for $h \in N_Y$.
\end{proof}

Applying Proposition \ref{prop:doublecoset}, we find the possible triples $(G, H_X, H_Y)$ such that moreover $H_X < H_Y < G$, all up to simultaneous conjugacy by $N_G$. If so desired, we can impose that $H_X$ be maximal in $H_Y$, to reflect that the corresponding map $X \to Y$ is indecomposable, and a similar remark applies to $H_Y$ and $G$.

It then remains to find the possible monodromy data $\Sigma$ starting from $(G, H_X, H_Y)$. For this, we have used fast and efficient code by Paulhus \cite{paulhus3} based on work of Breuer \cite{breuer}. This finds the possible $\Sigma$ up to conjugation by elements of $G$ once conjugacy classes in $G$ are given. While we do not have these conjugacy classes at our disposal, we do have imposed ramification data $R$, which above any point determines the cycle structure of the corresponding conjugacy classes (recall that our data furnish a conjugacy class of embeddings of $G$ into $S_d$, so that this is well-defined). This gives a finite number of explicit possibilities for the conjugacy classes above a given point. Combining the outcomes of Breuer's algorithms for all possible choices, we obtain the possible covers $\Sigma$. If so desired, we can still reduce the set of possible $\Sigma$ further under common $N_G$-conjugacy to prevent duplicates. While we usually do this, it can occasionally cost some time if there are lots of covers involved, in which case our algorithms allow this step to be skipped.

Finally, given an element $\Sigma$, we append it to one of the triples obtained before to obtain a quadruple $(G, H_X, H_Y, \Sigma)$. If this quadruple has the correct ramification, as can be checked using Algorithm \ref{algo:SecondTask}(i), then we retain this quadruple.

The above discussion motivates the following algorithm:

\begin{algorithm}\label{algo:FirstTask}
Input: $(g_X, g_Y, d_X, d_Y, R)$ as in Definition \ref{def:DiagramOfGivenType}.

Output: a list of 4-tuples $(G, H_X, H_Y, \Sigma)$.

Procedure:
\begin{enumerate}
  \item Initialize $d:=d_Xd_Y$ and let $\mathcal{L}_1$ and $\mathcal{L}_2$ be the empty lists.

  \item Loop over representatives $G$ of conjugacy classes of subgroups of $S_d$. For each representative do:
    \begin{enumerate}
      \item If $G$ is not transitive, discard $G$ and continue with the next subgroup;
      \item Set $H_X$ to be the stabilizer of $1$ in $G$;
      \item Append to $\mathcal{L}_1$ all triples $(G,H_X,H_Y)$ obtained using Proposition \ref{prop:doublecoset}.
    \end{enumerate}

  \item Using Breuer's algorithm as implemented by Paulhus, find all possible isomorphism classes of monodromy data $\Sigma$, up to $N_G$-conjugacy if desired. Loop over these $\Sigma$, and for a fixed such element do:
  \begin{enumerate}
    \item Loop over the triples $(G,H_X,H_Y)$ in $\mathcal{L}_1$;
    \item Using Algorithm \ref{algo:SecondTask}, compute the genera of $Z/H_X$ and of $Z/H_Y$. If $g(Z/H_X) \neq g_X$ or $g(Z/H_Y) \neq g_Y$, return to the beginning of the loop;
    \item Using Algorithm \ref{algo:SecondTask}, compute the ramification structure of $X \to \mathbb{P}^1$. If it is different from $R$, return to the beginning of the loop;
    \item Add $(G,H_X,H_Y,\Sigma)$ to $\mathcal{L}_2$.
  \end{enumerate}

  \item Return $\mathcal{L}_2$.
\end{enumerate}
\end{algorithm}

\subsubsection{Action of $G$ on $H^0(Z,\Omega^1_Z)$ and calculation of image dimensions}

Given a finite group $G$, one can compute its character table, for example by using the Dixon–Schneider algorithm, or the LLL-based induce/reduce algorithm of Unger \cite{unger}. Once the character table of $G$ is known, in order to fully describe the $G$-representation $V$ we simply need to determine the multiplicity with which each character $\chi$ of $G$ appears in $V$. Such multiplicities can be obtained by applying Theorem \ref{thm:Chevalley-Weil} to the map $Z \to \mathbb{P}^1$. Indeed, given a character $\chi$ corresponding to a representation $\tau_\chi$, the only information we need to determine $\nu_\chi$ are the numbers $e_i$ and $N_{i,\alpha}$. We have already observed that $e_i$ is the order of $\sigma_i$. Furthermore, by definition, $N_{i, \alpha}$ is the multiplicity of $\zeta_{e_i}$ as an eigenvalue of $\tau_{\chi}(\sigma_i)$. This multiplicity can be read off the characteristic polynomial of $\tau_\chi(\sigma_i)$, whose coefficients are the elementary symmetric functions in the eigenvalues of $\tau_\chi(\sigma_i)$. As we are in characteristic zero, the symmetric functions of $\lambda_1,\ldots,\lambda_k$ are determined by the Newton sums
\begin{equation*}
  \sum_{i=1}^k \lambda_i = \operatorname{tr} \tau_\chi(\sigma_i)= \chi(\sigma_i),\;  \sum_{i=1}^k \lambda_i^2 = \operatorname{tr} \tau_\chi(\sigma_i^2)= \chi(\sigma_i^2), \; \ldots, \; \sum_{i=1}^k \lambda_i^k = \operatorname{tr} \tau_\chi(\sigma_i^k)= \chi(\sigma_i^k).
\end{equation*}
This shows that the knowledge of the character $\chi$ is enough to determine the characteristic polynomial of $\tau_\chi(\sigma_i)$, hence we may compute the numbers $N_{i,\alpha}$ from the knowledge of $\chi$ without even having to describe the $G$-module $\tau_\chi$. This solves Part (ii) of Algorithm \ref{algo:SecondTask}. Part (iii) can then be obtained by calculating the relevant irreducible representations $\tau_\chi$ explicitly (there is functionality available to this end in our computer algebra system of choice \textsc{Magma}) and summing the dimensions of the images of the maps obtained in Proposition \ref{prop:projector}, multiplied by the relevant multiplicities. In our application, we often use the following more specific procedure:

\begin{algorithm}\label{algo:subroutine}
  Input: a 4-tuple $(G,H_X,H_Y,\Sigma)$ as in Problem \ref{problem:SecondTask}.

  Output: A subgroup $H_C$ of $G$ (if it exists) for which the corresponding curve $C = Z / H_C$ has the following properties:
  \begin{itemize}
    \item $0 < g_C \le g_X - g_Y$;
    \item The map $\Jac (C) \to \Jac (X)$ induced by $X \leftarrow Z \rightarrow C$ is injective;
    \item The image of $\Jac (C) \to \Jac (X)$ does not intersect the image of $\Jac (Y) \to \Jac (X)$.
  \end{itemize}

  Procedure:
  \begin{enumerate}
    \item Run through the subgroups of $H$ of $G$;
    \item Compute the genus $g_C$ of the curve $C:=Z/H$. If we do not have that $0 < g_C \le g_X - g_Y$, then move on to the next $H$, otherwise proceed to (iii);
    \item Compute the dimension of the image of the induced map $\Jac(C) \to \Jac(Y)$ using the $G$-module $H^0(Z,\Omega^1_Z)$ and Proposition \ref{prop:projector}. If it is non-zero, then move on to the next $H$, otherwise proceed to (iv);
    \item Using similar methods, compute the dimension of the image of the induced map $\Jac(C) \to \Jac(X)$. If its dimension does not equal $g_C$, then move on to the next $H$, otherwise return $H$.
  \end{enumerate}
\end{algorithm}

\begin{remark}
  Algorithm \ref{algo:subroutine} insists on the injectivity of the map $\Jac (C) \to \Jac (X)$ because otherwise we would have to deal with another decomposition problem in order to describe the part of the Prym variety thus obtained as a Jacobian.
\end{remark}

If the algorithm returns a group $H_C$ for which moreover $g_C = g_X - g_Y$, then $\Jac(X) \sim \Jac(C) \times \Jac(Y)$, so that (up to isogeny) we have realized the Prym variety of the cover $X \to Y$ as the Jacobian of the curve $C$. If for all $(G,H_X,H_Y,\Sigma)$ that we consider we can find a group $H_C$ and a corresponding curve $C$ as above, then we know that for every diagram $X \to Y \to \mathbb{P}^1$ of type $(g_X,g_Y,d_X,d_Y,R)$, the abelian variety $\operatorname{Prym}(X \to Y)$ is isogenous to the Jacobian of a quotient $C$ of the Galois closure of $X \to \mathbb{P}^1$. Even if this does not happen, it is still possible that we are successful for, say, all quadruples for which $G$ is in a certain specified isomorphism class. To our surprise, we have discovered several types $(g_X,g_Y,d_X,d_Y,R)$ for which this construction gives non-trivial information on the Prym variety, and we report on these findings in Section \ref{sec:Families} below.

Even if there is no single quotient $C$ in Algorithm \ref{algo:subroutine} such that $\Jac (C)$ is isogenous $to \operatorname{Prym}(X \to Y)$, it may still happen that the latter Prym variety is isogenous to a product of Jacobians obtained in this fashion, as can be ascertained by determining the sum of the corresponding subspaces in $\Jac (X)$. An example of this is given in the entry \texttt{rr-spec} of Table \ref{tab:results-g3}, as explained in Section \ref{sec:results}.

\subsubsection{Some fine print and speedups}

This final section contains a smattering of more detailed remarks on our implementation, calculations, and results. To start, we note that the calculation in Algorithm \ref{algo:subroutine}(ii) is possible from the knowledge of the modules $\tau_{\chi}$ and their multiplicities $n_{\chi}$, which we need only calculate once given $G$ and $\Sigma$. As we run through the possible $\Sigma$, we store the different intervening representations $\tau_{\chi}$ so that we do not have to recalculate them later for different $\Sigma$. (We do have to calculate new multiplicities $n_{\chi}$, but this is fortunately far less laborious.) This is worthwhile because our implementation works with the Chevalley-Weil decomposition throughout: All dimension calculations involving Proposition \ref{prop:projector} are done for the irreducible representation $\tau_{\chi}$, after which the corresponding results are summed with the relevant multiplicities $n_{\chi}$.

When looking for a single curve $C$ to furnish the complement of $\Jac (Y)$ in $\Jac (X)$, we can in fact do better than running over all possible $H$. Indeed, we still have the ambient isomorphism group $N_G$ to consider, and reasoning as in Proposition \ref{prop:doublecoset} shows that it suffices to consider candidate subgroups $K$ up to conjugacy at first. Whether the condition in Algorithm \ref{algo:subroutine}(ii) holds depends only on the $G$-conjugacy class of $H$. Given a representative $K$ of such a conjugacy class, the argument from Proposition \ref{prop:doublecoset} then shows that we only need consider the possibilities $(G, H_X, H_Y, \Sigma, n^{-1} K n)$ where $n$ runs through the double coset $N_C \backslash G / (N_X \cap N_Y)$ for $N_C$ (resp. $N_X$, $N_Y$) the normalizer of $K$ (resp.\ $H_X$, $H_Y$) in $G$. Since the pairs $(H_X, H_Y)$ in our quadruples $(G, H_X, H_Y, \Sigma)$ stem from a fixed list, we can store these double coset representatives on the fly so as not to have to recalculate them.

This same uniformity then ensures that only a relatively small number of triples $(H_X, H_Y, H)$ is encountered for a fixed group $G$, albeit for many different $\Sigma$ and with many different multiplicities. This makes it very worthwhile to store all the ranks and projectors in Proposition \ref{prop:projector} that are calculated when working with a fixed representation $\tau_{\chi}$ in a hash table, as considerable time is gained when using a lookup instead of a recalculation. In fact, in practice our calculations show that most time is spent constructing the explicit projectors in Proposition \ref{prop:projector} on the larger irreducible subrepresentations of $H^0(Z,\Omega^1_Z)$. Similarly, we can ensure that the rank of composition of these projectors does not need be calculated for $\tau_{\chi}$ when we encounter a previously stored triple $(H_X, H_Y, H)$, which is very often the case in practice.

%% file: results.tex

\section{Results}\label{sec:results}

\subsection{Presentation of the tables}
The tables in the appendix describe results obtained by running our algorithms. We recover \emph{all} classical results from the literature (up to time limitations of our codes), as we will discuss later on. First we explain how to read an entry in these tables, illustrated by the concrete case \texttt{total4}:

\begin{itemize}
  \item[``$g_X,g_Y,d_X$''] For the case \texttt{total4}, the genus $g_X$ of $X$ equals 4, the genus $g_Y$ of $Y$ equals $1$ and the degree $d_X$ of the cover $X \to Y$ equals $4$.
  \item[``Ramification''] This describes the ramification structure of the composition $X \to Y \to \P^1$.  The degree of $Y \to \P^1$ usually equals $2$. If not, the name of the case starts with the degree $\deg(Y/\P^1)$ (for instance \texttt{3-orig} in Table~\ref{tab:results-bruin}). A thin line represents an unramified point. A thick line without a number on its side a totally ramified point; for a thick line with a number on its side, this number specifies the ramification index. For all ramification types thus displayed, the number over them represents the number of copies in the total ramification structure. In the case \texttt{total4}, we see that the map $Y \to \P^1$ has $4$ ramification points, all of which split totally in the cover $X \to Y$. Moreover, the $2$ total ramification points of $X \to Y$ are merged under the map $X \to \P^1$.
  \item[``$\# G,g_Z$''] This lists the different possible pairs $\# G, g_Z$, where $\# G = \deg (Z \to \P^1)$ is the cardinality of the monodromy group $G$ and where $g_Z$ is the genus of $Z$. In the case \texttt{total4} there turns out to be only one such pair.
  \item[``$X$ nhyp/hyp''] Running through the possible cases from the previous item, we consider the isomorphism classes of curves $X$ for which an automorphism of the Galois closure induces a hyperelliptic involution. Given such a class, we use our algorithms to check whether a piece of the Prym variety of $X\to Y$ is given by the Jacobian of a quotient of $Z$. The number of curves for which this happens (resp.\ does not happen) is the second bracket entry of the case listed in this column. The first entry does the same, but instead for those isomorphism class of curves $X$ for which no hyperelliptic involution is induced by the Galois closure. In the case \texttt{total4}, we obtain 48 possibly non-hyperelliptic and 16 hyperelliptic curves in this way for the single possible pair $\# G, g_Z$, for all of which we can indeed generate a piece of the Prym variety as the Jacobian of a quotient of $Z$.
  \item[``Prym dims''] For the entries above, we give the dimensions of the disjoint pieces of the Prym variety that we found as Jacobians of quotients of $Z$, separated between non-hyperelliptic and hyperelliptic case (if one, or both, of these cases never yields a piece of the Prym, it does simply not appear). In the case \texttt{total4}, we always find a curve $C$ of genus $3$ such that $\Jac C \sim P(X/Y)$ in the non-hyperelliptic case. By contrast, in the hyperelliptic case, we find two curves $C_1$ and $C_2$ of genus $1$ and $2$ such that $\Jac C_1 \times \Jac C_2 \sim P(X/Y)$. It is possible that there are multiple cases with different resulting dimensions. This is illustrated in the case \texttt{total5}.
  \item[``$\deg Z \to C_i$''] The last column gives the degrees of the maps $Z \to C_i$ obtained in the previous entry, separated into the non-hyperelliptic and hyperelliptic case as before.
\end{itemize}

\begin{remark}
  Our implementation allows the determination of more information, like the ramification of intermediate covers.
\end{remark}

\begin{remark}
  Given a certain ramification structure, our programs can equally well calculate results for ``specializations'' of it, for instance those obtained by collapsing two ramification points of the cover $Y \to \P^1$. We did not try to do this systematically, but we did often observe that if one recovers the Prym as the Jacobian of a quotient of $Z$ in the generic initial case, this continues to hold for the specializations. A notable example of this is furnished by Table \ref{tab:results-bruin}.
\end{remark}

\subsection{Comments on the tables}
Let us start with examples that already appear in the existing literature and that we could recover and extend.
\begin{itemize}
  \item Table~\ref{tab:results-g3} recovers the \cite{RiRo} case we looked at in Section~\ref{sec:RiRoRedone} and which was the starting point of this article.
  \item Table~\ref{tab:results-g2} gathers covers of genus $2$ of curves genus $1$ by a map of degree $d_X$ with $2 \leq d_X \leq 7$. In all these cases the Prym (which is a curve of genus $1$) appears as a quotient of the Galois closure $Z$. Note that when $2 \leq d_X \leq 11$, there are direct construction of the Prym as an explicit curve of genus $1$: The case $d_X=2$ goes back to the work of Jacobi on abelian integrals, (see the references in \cite[p.395]{baker} or \cite{leprevost}), $d_X=3$ (see \cite{goursat}, \cite{kuhn} or the appendix of \cite{BHLS15}), $d_X=4$ (see \cite{bolzan4} and \cite{bruinn4}), $d_X=5$ (see \cite{MSV09}) and more generally when $d_X \leq 11$ (see \cite{kumar}).
  \item  Table~\ref{tab:results-dalaljan} gathers degree $2$-covers of hyperelliptic curves ramified  over exactly $2$ points. We recover the results of \cite{dalaljan} and \cite[Th.4.1]{levin}.
  \item Table~\ref{tab:results-bruin} gathers \'etale covers of degree $2$ of curves of genus $g_Y \geq 3$. Bruin's result \cite{bruin} for $g_X=5$ is the first configuration and we see that it seems to generalize well to higher genus. In this case, we have not determined all covers, as there are tens of millions of these, but have instead taken a sample of several hundreds of such covers by generating these randomly. Note that here, the map to $\P^1$ is of degree $3$, which for $g_Y=3$ and $g_Y=4$ is actually the smallest degree that is generically possible.
  \end{itemize}
  Here are some new ramification patterns:
  \begin{itemize}
  \item Table~\ref{tab:results-etg2} gathers \'etale covers of curves of genus $2$ by maps of degree $3,4$ and $5$. The situation is more chequered, since certain cases give positive results and other do not, even for the same ramification structure.
  \item Table~\ref{tab:results-total} is a new situation which does not appear in the literature. One sees that when the degree of the cover from $X$ to the curve $Y$ of genus $1$ is $3$ or $4$ we only get positive cases (that is, cases in which our strategy can describe the Prym as a Jacobian up to isogeny), but as soon as the degree is $5$, we only get very few favorable situations.
  \item Table~\ref{tab:results-g3deg4} gathers some miscellaneous cases.
\end{itemize}

%% file: families.tex

\subsection{Some explicit equations} \label{sec:Families}

As an application of our programs, we consider the first case of Table~\ref{tab:results-total}: $g_X=3$, $g_Y=1$, $d_X=3$ and $d_Y=2$ with the ramification data $$R=(\{(3,2)\},\{(2,3)\},\{(2,3)\},\{(2,3)\},\{(2,3)\}).$$
The cover $X \to \P^1$ may be Galois with $G \simeq S_3$: in this case, $X$ is non-hyperelliptic, and the result already appears in Table~\ref{table:NonHyper}. We therefore concentrate on the second case where $X$ is hyperelliptic. The programs show in the same way that $Z$ is also hyperelliptic, is equipped with an action of $C_2 \times S_3$, and admits both $X$ (of genus 3) and $C$ (of genus 2) as quotients by involutions. Since the action of $S_3$ commutes with that of the hyperelliptic involution, we may assume that the automorphism group is generated by $\sigma:\,(x,y)\mapsto(\zeta_3 x,y)$, $\tau:\,(x,y)\mapsto(1/x,y/x^6)$ and $\iota:\,(x,y)\mapsto(x,-y)$. This means that the hyperelliptic curve $Z$, of genus 5, is given by $$y^2=x^{12}+ax^9+bx^6+ax^3+1,$$ with discriminant $3^{12}\left(a^2-4 b+8\right)^6 \left((b+2)^2-4 a^2\right)^3\neq0$.
The quotient $Z/\langle \iota\tau \rangle$ is the genus 3 curve $X$, and by taking the fixed functions $u=x+1/x$ and $v=y(x-1/x)$ we obtain:
$$
X:\,v^2=(u^6-6u^4+au^3+9u^2-3au+(b-2))(u^2-4),
$$
with discriminant $-2^4\cdot 3^6 \cdot \left(a^2-4 b+8\right)^3 \left(4 a^2-(b+2)^2\right)^3\neq0$. The map $X \to Y$ can then be recovered by considering the quotient $Z/\langle \iota\tau, \sigma \rangle$: we then obtain $Y: t^2=(s^2-4)(s^2+as+(b-2))$ and a 3-to-1 map
\[
\begin{array}{ccc}
X & \to & Y \\
(u,v) & \mapsto & (u^3-3u, v(u^2-1)).
\end{array}
\]
The Prym variety of $X \to Y$ is isogenous to the Jacobian of $C=Y/\langle\tau \rangle$. By taking the fixed functions $u=x+1/x$ and $v'=y/x^3$ we obtain:
$$
C:\,v'^2=(u^6-6u^4+au^3+9u^2-3au+(b-2)).
$$
The discriminant is $-729 \left(a^2-4 b+8\right)^3 \left(4 a^2-(b+2)^2\right)\neq0$, so this is indeed a smooth curve of genus $2$.\\

We remark that once this example, and similar ones in higher genus, were brought to our attention by the output of our programs, we were able to spot a generalisation, which allowed us to recover some of the hyperelliptic cases in Table~\ref{tab:results-total}. Fix an integer $k \geq 2$ and consider the hyperelliptic curve $Z:\,y^2=f(x)$, where $f(x)=x^{4k}+ax^{3k}+bx^{2k}+ax^{k}+1$ for generic $a$ and $b$. Factor $f (x) =(x^k-\alpha_1^k)(x^k-\alpha_2^k)(x^k-\frac{1}{\alpha_1^k})(x^k-\frac{1}{\alpha_2^k})$. The automorphism group of $Z$ contains the hyperelliptic involution $\iota(x,y)=(x,-y)$ and the elements $\sigma(x,y)=(\zeta_k x,y)$ and $\tau(x,y)=(\frac{1}{x},\frac{y}{x^{2k}})$. The quotient $\pi_{Z/X}:\,Z\rightarrow X:=Z/\langle\iota\tau\rangle$ can be described thanks to the invariant functions $u=x+\frac{1}{x}$ and $v=\frac{y}{x^k}(x-\frac{1}{x})$, which lead to the equation
\[
X :\,v^2=(u^2-4)(g^2(u)+ag(u)+b-2).
\]
The polynomial $(u^2-4)(g^2(u)+ag(u)+b-2)$ has $2k+2$ different roots, and $X$ is a smooth hyperelliptic curve of genus $k$. Similarly, we consider  the quotient $\pi_{Z/C}:\,Z\rightarrow C:=Z/\langle\tau\rangle$, and to get an equation for $C$ we consider the invariant functions $u=x+\frac{1}{x}$ and $w=\frac{y}{x^k}$. We write first $f(x)=x^{2k}((x^k+\frac{1}{x^k})^2+a(x^k+\frac{1}{x^k})+b-2)$, and note that there exists a polynomial $g$ of degree $k$ such that $g(x+\frac{1}{x})=x^k+\frac{1}{x^k}$. We get then $C:\,w^2=g^2(u)+ag(u)+b-2$. The roots of $g^2(u)+ag(u)+b-2$ are $\zeta_k^i\alpha_j+\frac{1}{\zeta_k^i\alpha_j}$ with $i=0,1,..,k-1$ and $j=1,2$. This yields a smooth hyperelliptic curve $C$ of genus $g(C)=k-1$. Furthermore, we consider the quotient $\pi_{X/Y}:\,X\rightarrow Y=Z/\langle\iota\tau, \sigma \rangle$: we take the invariant functions $U=x^k+\frac{1}{x^k}=g(u)$ and $V=\frac{y}{x^k}(x^k-\frac{1}{x^k})=vh(u)$, where $h$ is a degree $k-1$ polynomial such that $h(x+\frac{1}{x})=\frac{x^k-\frac{1}{x^k}}{x-\frac{1}{x}}$. We get then the equation $Y:\,V^2=(U^2-4)(U^2+aU+b-2)$. The roots of $(U^2-4)(U^2+aU+b-2)$ are all different, and $Y$ is a smooth curve genus $1$. Computing the pullbacks to $Z$ of the regular differentials of $X$, $Y$ and $C$ leads to the following proposition:
\begin{proposition}\label{th:cased3g3}
  With the notation above, the Prym variety $\operatorname{Prym}(X/Y)$ is isogenous to the Jacobian of the curve $C$.
\end{proposition}

%% file: tables.tex
\begin{table}[htbp]
  \centering
  \begin{tiny}
    \begin{tabular}{|c|c|c|c|c|c|c|c|c|}
      \hline
      Case & $g_X,g_Y,d_X$ & Ramification & $\# G,g_Z$  & $X$ nhyp/hyp & Prym dims & $\deg Z \to C_i$ \\
      \hline
      \texttt{rr-spec} & $3,1,2$ &
      \begin{tabular}{c} $4$ \\ 
  \long\def\PICTURE #pictures.texram22 {}%
  \input{pictures.tex}
 \end{tabular}
      \begin{tabular}{c} $2$ \\ 
  \long\def\PICTURE #pictures.texramd {}%
  \input{pictures.tex}
 \end{tabular} &
      $4,3$ &
      $[3,0],[0,0]$ &
      $[1,1]$ &
      $[2]$ \\
      \hline
      \texttt{rr-gen} & $3,1,2$ &
      \begin{tabular}{c} $4$ \\ 
  \long\def\PICTURE #pictures.texram22 {}%
  \input{pictures.tex}
 \end{tabular}
      \begin{tabular}{c} $2$ \\ 
  \long\def\PICTURE #pictures.texram211 {}%
  \input{pictures.tex}
 \end{tabular}
      \begin{tabular}{c} $1$ \\ 
  \long\def\PICTURE #pictures.texramd {}%
  \input{pictures.tex}
 \end{tabular} &
      $8,7$ &
      $[8,0],[0,0]$ &
      $[2]$ &
      $[2]$ \\
      \hline
    \end{tabular}
  \end{tiny}
  \caption{Recovering Ritzenthaler--Romagny}
  \label{tab:results-g3}
\end{table}

\begin{table}[htbp]
  \centering
  \begin{tiny}
    \begin{tabular}{|c|c|c|c|c|c|c|c|c|}
      \hline
      Case & $g_X,g_Y,d_X$ & Ramification & $\# G,g_Z$  & $X$ nhyp/hyp & Prym dims & $\deg Z \to C_i$ \\
      \hline
      \texttt{g2-2} & $2,1,2$ &
      \begin{tabular}{c} $4$ \\ 
  \long\def\PICTURE #pictures.texram22 {}%
  \input{pictures.tex}
 \end{tabular}
      \begin{tabular}{c} $1$ \\ 
  \long\def\PICTURE #pictures.texramd {}%
  \input{pictures.tex}
 \end{tabular} &
      $4,2$ &
      [0,0],[4,0] &
      [1] &
      [2] \\
      \hline
      \texttt{g2-3} & $2,1,3$ &
      \begin{tabular}{c} $4$ \\ 
  \long\def\PICTURE #pictures.texram22 {}%
  \input{pictures.tex}
 \end{tabular}
      \begin{tabular}{c} $1$ \\ 
  \long\def\PICTURE #pictures.texramd {}%
  \input{pictures.tex}
 \end{tabular} &
      $12,4$ &
      [0,0],[16,0] &
      [1] &
      [2] \\
      \hline
      \texttt{g2-4} & $2,1,4$ &
      \begin{tabular}{c} $4$ \\ 
  \long\def\PICTURE #pictures.texram22 {}%
  \input{pictures.tex}
 \end{tabular}
      \begin{tabular}{c} $1$ \\ 
  \long\def\PICTURE #pictures.texramd {}%
  \input{pictures.tex}
 \end{tabular} &
      $48,13$ &
      [0,0],[48,0] &
      [1] &
      [4] \\
      \hline
      \texttt{g2-5} & $2,1,5$ &
      \begin{tabular}{c} $4$ \\ 
  \long\def\PICTURE #pictures.texram22 {}%
  \input{pictures.tex}
 \end{tabular}
      \begin{tabular}{c} $1$ \\ 
  \long\def\PICTURE #pictures.texramd {}%
  \input{pictures.tex}
 \end{tabular} &
      $240,61$ &
      [0,0],[160,0] &
      [1] &
      [12] \\
      \hline
      \texttt{g2-6} & $2,1,6$ &
      \begin{tabular}{c} $4$ \\ 
  \long\def\PICTURE #pictures.texram22 {}%
  \input{pictures.tex}
 \end{tabular}
      \begin{tabular}{c} $1$ \\ 
  \long\def\PICTURE #pictures.texramd {}%
  \input{pictures.tex}
 \end{tabular} &
      $1440,361$ &
      [0,0],[240,0] &
      [1] &
      [48] \\
      \hline
      \texttt{g2-7} & $2,1,7$ &
      \begin{tabular}{c} $4$ \\ 
  \long\def\PICTURE #pictures.texram22 {}%
  \input{pictures.tex}
 \end{tabular}
      \begin{tabular}{c} $1$ \\ 
  \long\def\PICTURE #pictures.texramd {}%
  \input{pictures.tex}
 \end{tabular} &
      $10080,2521$ &
      [0,0],[672,0] &
      [1] &
      [240] \\
      \hline
    \end{tabular}
  \end{tiny}
  \caption{Genus $2$ to genus $1$}
  \label{tab:results-g2}
\end{table}

\begin{table}[htbp]
  \centering
  \begin{tiny}
    \begin{tabular}{|c|c|c|c|c|c|c|c|c|}
      \hline
      Case & $g_X,g_Y,d_X$ & Ramification & $\# G,g_Z$  & $X$ nhyp/hyp & Prym dims & $\deg Z \to C_i$ \\
      \hline
      \texttt{g2} & $4,2,2$ &
      \begin{tabular}{c} $6$ \\ 
  \long\def\PICTURE #pictures.texram22 {}%
  \input{pictures.tex}
 \end{tabular}
      \begin{tabular}{c} $2$ \\ 
  \long\def\PICTURE #pictures.texram211 {}%
  \input{pictures.tex}
 \end{tabular} &
      $8,9$ &
      $[32,0],[0,0]$ &
      $[2]$ &
      $[2]$ \\
      \hline
      \texttt{g3} & $6,3,2$ &
      \begin{tabular}{c} $8$ \\ 
  \long\def\PICTURE #pictures.texram22 {}%
  \input{pictures.tex}
 \end{tabular}
      \begin{tabular}{c} $2$ \\ 
  \long\def\PICTURE #pictures.texram211 {}%
  \input{pictures.tex}
 \end{tabular} &
      $8,13$ &
      $[128,0],[0,0]$ &
      $[3]$ &
      $[2]$ \\
      \hline
      \texttt{g4} & $8,4,2$ &
      \begin{tabular}{c} $10$ \\ 
  \long\def\PICTURE #pictures.texram22 {}%
  \input{pictures.tex}
 \end{tabular}
      \begin{tabular}{c} $2$ \\ 
  \long\def\PICTURE #pictures.texram211 {}%
  \input{pictures.tex}
 \end{tabular} &
      $8,17$ &
      $[512,0],[0,0]$ &
      $[4]$ &
      $[2]$ \\
      \hline
    \end{tabular}
  \end{tiny}
  \caption{Testing Dalaljan's results}
  \label{tab:results-dalaljan}
\end{table}

\begin{table}[htbp]
  \centering
  \begin{tiny}
    \begin{tabular}{|c|c|c|c|c|c|c|c|c|}
      \hline
      Case & $g_X,g_Y,d_X$ & Ramification & $\# G,g_Z$  & $X$ nhyp/hyp & Prym dims & $\deg Z \to C_i$ \\
      \hline
      \texttt{3-orig} & $5,3,2$ &
      \begin{tabular}{c} $10$ \\ 
  \long\def\PICTURE #pictures.texrambruin {}%
  \input{pictures.tex}
 \end{tabular} &
      $24,37$ &
      $[?,0],[0,0]$ &
      $[2]$ &
      $[6]$ \\
      \hline
      \texttt{3-g7} & $7,4,2$ &
      \begin{tabular}{c} $12$ \\ 
  \long\def\PICTURE #pictures.texrambruin {}%
  \input{pictures.tex}
 \end{tabular} &
      $24,49$ &
      $[?,0],[0,0]$ &
      $[2]$ &
      $[6]$ \\
      \hline
      \texttt{3-g9} & $9,5,2$ &
      \begin{tabular}{c} $14$ \\ 
  \long\def\PICTURE #pictures.texrambruin {}%
  \input{pictures.tex}
 \end{tabular} &
      $24,61$ &
      $[?,0],[0,0]$ &
      $[2]$ &
      $[6]$ \\
      \hline
      \texttt{3-g11} & $11,6,2$ &
      \begin{tabular}{c} $16$ \\ 
  \long\def\PICTURE #pictures.texrambruin {}%
  \input{pictures.tex}
 \end{tabular} &
      $24,73$ &
      $[?,0],[0,0]$ &
      $[2]$ &
      $[6]$ \\
      \hline
    \end{tabular}
  \end{tiny}
  \caption{Generalizing Bruin's results}
  \label{tab:results-bruin}
\end{table}

\begin{table}[htbp]
  \centering
  \begin{tiny}
    \begin{tabular}{|c|c|c|c|c|c|c|c|c|}
      \hline
      Case & $g_X,g_Y,d_X$ & Ramification & $\# G,g_Z$  & $X$ nhyp/hyp & Prym dims & $\deg Z \to C_i$ \\
      \hline
      \texttt{etale-3} & $4,2,3$ &
      \begin{tabular}{c} $6$ \\ 
  \long\def\PICTURE #pictures.texram23 {}%
  \input{pictures.tex}
 \end{tabular} &
      \begin{tabular}{c} $6,4$ \\ $12, 5$ \end{tabular} &
      \begin{tabular}{c} $[40,0],[0,0]$ \\ $[0,0],[60,0]$ \end{tabular} &
      \begin{tabular}{c} $[1,1]$ \\ $[2]$ \end{tabular} &
      \begin{tabular}{c} $[2,2]$ \\ $[2]$   \end{tabular} \\
      \hline
      \texttt{etale-4} & $5,2,4$ &
      \begin{tabular}{c} $6$ \\ 
  \long\def\PICTURE #pictures.texram24 {}%
  \input{pictures.tex}
 \end{tabular} &
      \begin{tabular}{c} $24,13$ \\ $48, 23$ \\ $96, 49$ \end{tabular} &
      \begin{tabular}{c} $[0,0],[0,120]$ \\ $[540,0],[360,0]$ \\ $[240,0],[0,0]$ \end{tabular} &
      \begin{tabular}{c} \phantom{$0$} \\ $[3], [3]$ \\ $[3]$ \end{tabular} &
      \begin{tabular}{c} \phantom{$0$} \\ $[8], [4]$ \\ $[6]$ \end{tabular} \\
      \hline
      \texttt{etale-5} & $6,2,5$ &
      \begin{tabular}{c} $6$ \\ 
  \long\def\PICTURE #pictures.texram25 {}%
  \input{pictures.tex}
 \end{tabular} &
      \begin{tabular}{c} $10,6$ \\ $20,11$ \\ $120,61$ \\ $200,101$ \\ $240,121$ \\ $14400,7201$ \end{tabular} &
      \begin{tabular}{c} $[156,0],[0,0]$ \\ $[0,0],[180,0]$ \\ $[4032,0],[0,0]$ \\ $[0,2785],[0,0]$ \\ $[0,10080],[0,6000]$ \\ no good \end{tabular} &
      \begin{tabular}{c}  $[2,2]$ \\ $[4]$ \\ $[2]$ or $[1]$ \\ \phantom{$0$} \\ \phantom{$0$} \\ \end{tabular} &
      \begin{tabular}{c}  $[2,2]$ \\ $[2]$ \\ $[24]$ or $[8]$ \\ \phantom{$0$} \\ \phantom{$0$} \\ \end{tabular} \\
      \hline
    \end{tabular}
  \end{tiny}
  \caption{Postcomposing étale covers of curves of genus $2$}
  \label{tab:results-etg2}
\end{table}

\begin{table}[htbp]
  \centering
  \begin{tiny}
    \begin{tabular}{|c|c|c|c|c|c|c|c|c|}
      \hline
      Case & $g_X,g_Y,d_X$ & Ramification & $\# G,g_Z$  & $X$ nhyp/hyp & Prym dims & $\deg Z \to C_i$ \\
      \hline
      \texttt{total-3} & $3,1,3$ & \begin{tabular}{cc} 4 & 1 \\ 
  \long\def\PICTURE #pictures.texram23 {}%
  \input{pictures.tex}
 & 
  \long\def\PICTURE #pictures.texramd {}%
  \input{pictures.tex}
 \end{tabular} & \begin{tabular}{c} $6,3$ \\ $12,5$ \end{tabular} & \begin{tabular}{c} $[9,0],[0,0]$ \\ $[0,0],[6,0]$ \end{tabular} & \begin{tabular}{c} $[2]$ \\ $[2]$ \end{tabular} & \begin{tabular}{c} $[2]$ \\ $[2]$ \end{tabular} \\
      \hline
      \texttt{total-4} & $4,1,4$ & \begin{tabular}{cc} 4 & 1 \\ 
  \long\def\PICTURE #pictures.texram24 {}%
  \input{pictures.tex}
 & 
  \long\def\PICTURE #pictures.texramd {}%
  \input{pictures.tex}
 \end{tabular} & $48,19$  & $[48,0],[16,0]$ & $[3],[3]$ & $[8],[4]$ \\
      \hline
      \texttt{total-5} & $5,1,5$ & \begin{tabular}{cc} 4 & 1 \\ 
  \long\def\PICTURE #pictures.texram25 {}%
  \input{pictures.tex}
 & 
  \long\def\PICTURE #pictures.texramd {}%
  \input{pictures.tex}
 \end{tabular} &
      \begin{tabular}{c} $10,5$ \\ $20, 9$ \\ $100, 41$ \\ $120, 49$ \\ $120, 49$ \\ $200, 81$ \\ $240, 97$ \\ $7200, 2881$ \\ $14400, 5761$ \end{tabular} &
      \begin{tabular}{c} $[25,0],[0,0]$ \\ $[0,0],[12,0]$ \\ $[0,12],[0,0]$ \\ $[380,0],[0,0]$ \\ $[0,0],[0,25]$ \\ $[0,96],[0,0]$ \\ $[0,450],[0,60]$ \\ $[0,90],[0,0]$ \\ $[0,360],[0,0]$ \end{tabular} &
      \begin{tabular}{c} $[4]$ \\ $[2]$ \\ \phantom{$0$} \\ $[2]$ or $[1]$  \\ \phantom{$0$} \\ \phantom{$0$} \\ \phantom{$0$} \\ \phantom{$0$} \\ \phantom{$0$} \end{tabular} &
      \begin{tabular}{c} $[2]$ \\ $[2]$ \\ \phantom{$0$} \\ $[24]$ or $[8]$ \\ \phantom{$0$} \\ \phantom{$0$} \\ \phantom{$0$} \\ \phantom{$0$} \\ \phantom{$0$} \end{tabular} \\
      \hline
    \end{tabular}
  \end{tiny}
  \caption{Merging total ramification above two points}
  \label{tab:results-total}
\end{table}

\begin{table}[htbp]
  \centering
  \begin{tiny}
    \begin{tabular}{|c|c|c|c|c|c|c|c|c|}
      \hline
      Case & $g_X,g_Y,d_X$ & Ramification & $\# G,g_Z$  & $X$ nhyp/hyp & Prym dims & $\deg Z \to C_i$ \\
      \hline
      \texttt{3131} & $3,1,4$ &
      \begin{tabular}{c} $4$ \\ 
  \long\def\PICTURE #pictures.texram24 {}%
  \input{pictures.tex}
 \end{tabular}
      \begin{tabular}{c} $1$ \\ 
  \long\def\PICTURE #pictures.texramp3 {}%
  \input{pictures.tex}
 \end{tabular} &
      \begin{tabular}{c} $24,9$ \\ $48,17$ \\ $96,33$ \end{tabular} &
      \begin{tabular}{c} $[0,0],[0,27]$ \\ $[54,0],[36,0]$  \\ $[54,0],[0,0]$ \end{tabular} &
      \begin{tabular}{c} \phantom{$0$} \\ $[2],[1,1]$ \\ $[2]$ \end{tabular} &
      \begin{tabular}{c} \phantom{$0$} \\ $[8],[4,6]$ \\ $[6]$ \end{tabular} \\
      \hline
      \texttt{4211} & $3,1,4$ &
      \begin{tabular}{c} $4$ \\ 
  \long\def\PICTURE #pictures.texram24 {}%
  \input{pictures.tex}
 \end{tabular}
      \begin{tabular}{c} $1$ \\ 
  \long\def\PICTURE #pictures.texramp2d {}%
  \input{pictures.tex}
 \end{tabular} &
      $192,73$ &
      $[64,0],[0,0]$ &
      $[2]$ &
      $[16]$ \\
      \hline
    \end{tabular}
  \end{tiny}
  \caption{Genus $3$ to genus $1$, degree $4$}
  \label{tab:results-g3deg4}
\end{table}

%% file: pictures.tex
\PICTURE ram10
\begin{tikzpicture}[line cap=round,line join=round,>=triangle 45,x=1cm,y=1cm,xscale=0.7,yscale=0.7]

\draw [line width=2pt] (0,.7) -- (0,2);
\draw [line width=2pt] (0,0) -- (0,.7);

\end{tikzpicture}

\ENDPICTURE

\PICTURE ram22thin
\begin{tikzpicture}[line cap=round,line join=round,>=triangle 45,x=1cm,y=1cm,xscale=0.7,yscale=0.7]

\draw [-] (0,.7) -- (-0.3,2);
\draw [-] (0,.7) -- (0.3,2);
\draw [-] (0,0)-- (0,.7);

\end{tikzpicture}
\ENDPICTURE

\PICTURE ram22thick
\begin{tikzpicture}[line cap=round,line join=round,>=triangle 45,x=1cm,y=1cm,xscale=0.7,yscale=0.7]

\draw [-] (0,.7) -- (-0.3,2);
\draw [-] (0,.7) -- (0.3,2);
\draw [line width=2pt] (0,0)-- (0,.7);
\draw (0.6,0.3) node {$2$};

\end{tikzpicture}
\ENDPICTURE

\PICTURE ram22
\begin{tikzpicture}[line cap=round,line join=round,>=triangle 45,x=1cm,y=1cm,xscale=0.7,yscale=0.7]

\draw [-] (0,.7) -- (-0.3,2);
\draw [-] (0,.7) -- (0.3,2);
\draw [line width=2pt] (0,0)-- (0,.7);

\end{tikzpicture}
\ENDPICTURE

\PICTURE ram23
\begin{tikzpicture}[line cap=round,line join=round,>=triangle 45,x=1cm,y=1cm,xscale=0.7,yscale=0.7]

\draw [-] (0,.7) -- (-0.4,2);
\draw [-] (0,.7) -- (0,2);
\draw [-] (0,.7) -- (0.4,2);
\draw [line width=2pt] (0,0)-- (0,.7);

\end{tikzpicture}
\ENDPICTURE

\PICTURE ram24
\begin{tikzpicture}[line cap=round,line join=round,>=triangle 45,x=1cm,y=1cm,xscale=0.7,yscale=0.7]

\draw [-] (0,.7) -- (-0.6,2);
\draw [-] (0,.7) -- (-0.2,2);
\draw [-] (0,.7) -- (0.2,2);
\draw [-] (0,.7) -- (.6,2);
\draw [line width=2pt] (0,0)-- (0,.7);

\end{tikzpicture}
\ENDPICTURE

\PICTURE ram25
\begin{tikzpicture}[line cap=round,line join=round,>=triangle 45,x=1cm,y=1cm,xscale=0.7,yscale=0.7]

\draw [-] (0,.7) -- (-0.6,2);
\draw [-] (0,.7) -- (-0.3,2);
\draw [-] (0,.7) -- (0,2);
\draw [-] (0,.7) -- (0.3,2);
\draw [-] (0,.7) -- (.6,2);
\draw [line width=2pt] (0,0)-- (0,.7);

\end{tikzpicture}
\ENDPICTURE

\PICTURE ramd
\begin{tikzpicture}[line cap=round,line join=round,>=triangle 45,x=1cm,y=1cm,xscale=0.7,yscale=0.7]

\draw [-] (0,0) -- (-0.3,0.7);
\draw [-] (0,0) -- (0.3,0.7);
\draw [line width=2pt] (-0.3,0.7)-- (-0.3,2);
\draw [line width=2pt] (0.3,0.7)-- (0.3,2);

\end{tikzpicture}
\ENDPICTURE

\PICTURE ramp2
\begin{tikzpicture}[line cap=round,line join=round,>=triangle 45,x=1cm,y=1cm,xscale=0.7,yscale=0.7]

\draw [-] (0,0) -- (-0.3,0.7);
\draw [-] (0,0) -- (0.3,0.7);
\draw [-] (-0.3,0.7)-- (-0.3,2);
\draw [-] (-0.3,0.7)-- (-0.2,2);
\draw [-] (-0.3,0.7)-- (-0.1,2);
\draw [line width=2pt] (-0.3,0.7)-- (-0.5,2);
\draw [line width=2pt] (0.3,0.7)-- (0.5,2);
\draw [-] (0.3,0.7)-- (0.1,2);
\draw [-] (0.3,0.7)-- (0.2,2);
\draw [-] (0.3,0.7)-- (0.3,2);
\draw (-0.8,1.5) node {$2$};
\draw (0.8,1.5) node {$2$};
\end{tikzpicture}
\ENDPICTURE

\PICTURE ramp2d
\begin{tikzpicture}[line cap=round,line join=round,>=triangle 45,x=1cm,y=1cm,xscale=0.7,yscale=0.7]

\draw [-] (0,0) -- (-0.3,0.7);
\draw [-] (0,0) -- (0.3,0.7);
\draw [line width=2pt] (-0.3,0.7)-- (-0.3,2);
\draw [line width=2pt] (0.3,0.7)-- (0.5,2);
\draw [-] (0.3,0.7)-- (0.0,2);
\draw [-] (0.3,0.7)-- (0.2,2);
\draw (-0.8,1.5) node {$4$};
\draw (0.8,1.5) node {$2$};
\end{tikzpicture}
\ENDPICTURE

\PICTURE ramp3
\begin{tikzpicture}[line cap=round,line join=round,>=triangle 45,x=1cm,y=1cm,xscale=0.7,yscale=0.7]

\draw [-] (0,0) -- (-0.3,0.7);
\draw [-] (0,0) -- (0.3,0.7);
\draw [-] (-0.3,0.7)-- (-0.2,2);
\draw [line width=2pt] (-0.3,0.7)-- (-0.4,2);
\draw [line width=2pt] (0.3,0.7)-- (0.4,2);
\draw [-] (0.3,0.7)-- (0.2,2);
\draw (-0.6,1.5) node {$3$};
\draw (0.6,1.5) node {$3$};
\end{tikzpicture}
\ENDPICTURE

\PICTURE ramp4
\begin{tikzpicture}[line cap=round,line join=round,>=triangle 45,x=1cm,y=1cm,xscale=0.7,yscale=0.7]

\draw [-] (0,0) -- (-0.3,0.7);
\draw [-] (0,0) -- (0.3,0.7);
\draw [-] (-0.3,0.7)-- (-0.2,2);
\draw [line width=2pt] (-0.3,0.7)-- (-0.4,2);
\draw [line width=2pt] (0.3,0.7)-- (0.4,2);
\draw [-] (0.3,0.7)-- (0.2,2);
\draw (-0.6,1.5) node {$4$};
\draw (0.6,1.5) node {$4$};
\end{tikzpicture}
\ENDPICTURE

\PICTURE rambruin
\begin{tikzpicture}[line cap=round,line join=round,>=triangle 45,x=1cm,y=1cm,xscale=0.7,yscale=0.7]

\draw [-] (0,0) -- (-0.3,0.7);
\draw [line width=2pt] (0,0) -- (0.3,0.7);
\draw (0.6,0.4) node {$2$};
\draw [-] (-0.3,0.7)-- (-0.2,2);
\draw [-] (-0.3,0.7)-- (-0.4,2);
\draw [-] (0.3,0.7)-- (0.4,2);
\draw [-] (0.3,0.7)-- (0.2,2);
\end{tikzpicture}
\ENDPICTURE

\PICTURE ram211
\begin{tikzpicture}[line cap=round,line join=round,>=triangle 45,x=1cm,y=1cm,xscale=0.7,yscale=0.7]

\draw [-] (0,0) -- (-0.3,0.7);
\draw [-] (0,0) -- (0.3,0.7);
\draw [line width=2pt] (-0.3,0.7) -- (-0.3,2);
\draw [-] (0.3,0.7) -- (0.2,2);
\draw [-] (0.3,0.7) -- (0.4,2);
\end{tikzpicture}
\ENDPICTURE

\PICTURE ram311
\begin{tikzpicture}[line cap=round,line join=round,>=triangle 45,x=1cm,y=1cm,xscale=0.7,yscale=0.7]

\draw [-] (0,0) -- (-0.3,0.7);
\draw [-] (0,0) -- (0.3,0.7);
\draw [-] (-0.3,0.7)-- (-0.1,2);
\draw [-] (-0.3,0.7)-- (-0.2,2);
\draw [line width=2pt] (-0.3,0.7)-- (-0.4,2);
\draw [line width=2pt] (0.3,0.7)-- (0.4,2);
\draw [-] (0.3,0.7)-- (0.2,2);
\draw [-] (0.3,0.7)-- (0.1,2);
\draw (-0.6,1.5) node {$3$};
\draw (0.6,1.5) node {$3$};
\end{tikzpicture}
\ENDPICTURE

\PICTURE ram511111
\begin{tikzpicture}[line cap=round,line join=round,>=triangle 45,x=1cm,y=1cm,xscale=0.7,yscale=0.7]

\draw [-] (0,0) -- (-0.3,0.7);
\draw [-] (0,0) -- (0.3,0.7);
\draw [line width=2pt] (-0.3,0.7) -- (-0.3,2);
\draw [-] (0.3,0.7) -- (0.2,2);
\draw [-] (0.3,0.7) -- (0.3,2);
\draw [-] (0.3,0.7) -- (0.4,2);
\draw [-] (0.3,0.7) -- (0.5,2);
\draw [-] (0.3,0.7) -- (0.6,2);
\end{tikzpicture}
\ENDPICTURE